\documentclass[draftclsnofoot,onecolumn]{IEEEtran}

\usepackage{amssymb,amsmath,relsize,mathrsfs,amsmath,booktabs,epsfig,url}
\usepackage{graphics,hyperref,color}
\usepackage{algorithm,algorithmic}

\def\x{\mathbf x}
\def\y{\mathbf y}
\def\z{\mathbf z}
\def\X{\mathbf X}
\def\Y{\mathbf Y}
\def\Z{\mathbf Z}
\def\R{\mathbb R}
\def\U{\mathbf U}
\def\V{\mathbf V}
\def\v{\mathbf v}
\def\I{\mathbf I}
\def\E{\mathbf E}
\def\d{\mathbf d}
\def\D{\mathbf D}
\def\n{\mathbf n}

\def\B{\mathbf B}

\def\Sparsifier{{\mathcal T}}
\def\SignalClass{{\mathcal X}}
\newtheorem{theorem}{\textit{Theorem}}
\newtheorem{remark}{\textit{Remark}}
\newtheorem{lem}{\textit{Lemma}}
\newtheorem{proposition}{\textit{Proposition}}

\newtheorem{definition}{\textit{Definition}}
\newcommand{\sgn}{\operatorname{sgn}}
\newcommand{\esgn}{\overline{\operatorname{sgn}}}
\newcommand{\tr}{\operatorname{tr}}

\newcommand{\argmin}{\operatorname{argmin}}

\newcommand{\subto}{\operatorname{s.t.}}
\newcommand{\vect}{\operatorname{vect}}

\newcommand{\Omg}{\mathbf{\Omega}}
\newcommand{\omg}{\mathbf{\omega}}
\newcommand{\norm}[1]{\|#1\|}

\begin{document}

\title{Constrained Overcomplete Analysis Operator Learning for Cosparse Signal Modelling}

\author{Mehrdad~Yaghoobi, Sangnam Nam, R{\'e}mi Gribonval and Mike E. Davies
        
\thanks{This work was supported by the EU FP7, FET-Open grant number 225913, the European Research Council - PLEASE project under grant ERC-StG-2011-277906 and EPSRC grants EP/J015180/1 and EP/F039697/1. MED acknowledges support of his position from the Scottish Funding Council and their support of the Joint Research Institute with the Heriot-Watt University as a component part of the Edinburgh Research Partnership.

}

}

\maketitle

\begin{abstract}
We consider the problem of learning a low-dimensional signal model from a collection of training samples. The mainstream approach would be to learn an overcomplete dictionary to provide good approximations of the training samples using sparse synthesis coefficients. This famous sparse model has a less well known counterpart, in analysis form, called the cosparse analysis model. In this new model, signals are characterised by their parsimony in a transformed domain using an overcomplete (linear) {\em analysis operator}. We propose to learn an analysis operator from a training corpus using a \emph{constrained} optimisation framework based on L1 optimisation. The reason for introducing a constraint in the optimisation framework is to exclude trivial solutions. Although there is no final answer here for which constraint is the most relevant constraint, we investigate some conventional constraints in the model adaptation field and use the uniformly normalised tight frame (UNTF) for this purpose. We then derive a practical learning algorithm, based on projected subgradients and Douglas-Rachford splitting technique, and demonstrate its ability to robustly recover a ground truth analysis operator, when provided with a clean training set, of sufficient size. We also find an analysis operator for images, using some noisy cosparse signals, which is indeed a more realistic experiment. As the derived optimisation problem is not a convex program, we often find a local minimum using such variational methods. For two different settings, we provide preliminary theoretical support for the well-posedness of the learning problem, which can be practically used to test the local identifiability conditions of learnt operators.
\end{abstract}


\IEEEpeerreviewmaketitle

\section{Introduction}
\IEEEPARstart{M}{any} natural signals can be modelled using low-dimensional structures. This means, although the investigated signals are distributed in a high dimensional space, they can be represented using just a few parameters in an appropriate model \cite{Baraniuk10}. This is made possible by the fact that few parameters are actually involved in generating/describing such signals. Examples include polyphonic music, speech signals, cartoon images (where some lines/edges describe the images) and standard videos, where the scenes change slightly frame to frame. In modern signal processing, low-dimensional modelling is one of the most effective approaches to clarify the ambiguities in inverse problems, regularise the signals and overcome some conventional physical, temporal and computational barriers. This model has been heavily investigated in the last decade and many remarkable results achieved in almost all signal processing applications, see for example \cite{Starck10,Elad10}.

A natural and crucial question immediately arises. Given a class of signals that we are interested in, how can we model its low-dimensional structures? As an abstract answer, one could imagine that if $\SignalClass \subset \R^{n}$ is the signal class with low-dimensional structures, then there must be a map $\Sparsifier: \R^{n} \to \R^{a}$ such that $\Sparsifier (\x)$ is sparse for all $\x \in \SignalClass$. Unfortunately, finding such $\Sparsifier$ out of all possible
maps would be too complex, and one would have to restrict admissible class of maps. A simple option is to consider the class of linear maps. Hence, we ask: is there $\Omg \in \R^{a\times n}$ such that
$\Omg\x$ is sparse for $\x \in \SignalClass$? This is the type of problem we focus on in this paper. Before getting to the nitty-gritty, we review relevant results to give some contexts to our approach.

\subsection{Sparse Signal Models}
The most familiar type of low-dimensional structure is the sparse synthesis model \cite{Mallat93,Chen98}. In this setting, if a signal $\x \in \mathbb{R}^{n}$ can approximately be generated by adding just a few elementary vectors $\{ \d_{i}\}$, called atoms, from an overcomplete set, called a dictionary, as follows,

\begin{equation}\label{eq:lgm}
 \x \approx \sum_{\overset{i \in \mathcal{I}}{v_{i} \ne 0}}v_{i} \d_{i} = \D \v,
\end{equation}
 where $\D \in \R^{n \times p}$ is the dictionary (matrix) and $\mathcal{I}$ is an index set with small cardinality, \textit{i.e.} $|\mathcal{I}| \ll n$, it is called \textit{(approximately) sparse in $\mathbf{D}$}. From the signal processing perspective, it is then important to identify the support $\mathcal{I}$. However, this is not an easy task \cite{Davis94,Natarajan95}: when the signals are at most $k$-sparse---this means $|\mathcal{I}| \le k$---in model (\ref{eq:lgm}), they live in a union of subspaces (UoS) model \cite{Lu08,Blumensath09a,Eldar09} where each subspace has a dimension of at most $k$, and the number of such subspaces is exponential. A popular and practical approach to find a sparse $\v$ is the $\ell_1$-minimisation \cite{Chen98} in which one looks for $\v$ that minimises $\norm{\v}_1$ among the ones that satisfy $\x \approx \D\v$. The attractiveness of the $\ell_1$-objective comes from the fact that it promotes sparsity and is convex---good in computational aspects---at the same time.

An alternative form of $\ell_1$-minimisation has been used in practice. In this form, one looks for $\x$ that minimises $\norm{\Omg\x}_1$ among the signals that matches available (linear) information. Here, $\Omg\in\R^{a\times n}$ is some known linear operator. Note that a solution of this $\ell_1$-minimisation necessarily leads to many zeros in $\Omg\x$ and hence is orthogonal to many rows of $\Omg$ (see Figure \ref{fig:failure}.a). This observation is not exactly compatible with the synthesis interpretation above, and the corresponding low-dimensional model is called the \emph{cosparse analysis model} \cite{Nam11b}. A signal $\x$ in this model is called \emph{cosparse}. $\Omg$ is referred to as the \emph{analysis operator} in this paper. Although this model is also an instance of the UoS model and relies on the vector form of parsimony, it has many structural differences with the synthesis sparsity model \cite{Elad07c,Nam11b}. Very recently the signal recovery using analysis sparsity model has been investigated \cite{Candes11,Nam11b,Vaiter11}. As can be deduced from the expression $\Omg\x$, the cosparse analysis model is intimately related to our work.

\subsection{Model Adaptation}
As can be seen from the synthesis model description, the dictionary $\D$ plays an important role.
When we deal with the natural signals, one can then naturally ask, how good we can choose such a dictionary? 
Generally, we can use some signal exemplars \cite{Olshausen97,Lewicki00} for learning a dictionary or some domain knowledge for designing a suitable dictionary \cite{Donoho01a,Yaghoobi09e}. The readers can refer to \cite{Rubinstein10} and \cite{Tosic11} as some surveys on the dictionary selection approaches. It has been shown that the learned dictionaries usually out perform the canonical dictionaries, \textit{e.g.} DCT and Wavelet, in practice. Online dictionary learning methods \cite{Mairal10} have also been introduced to adapt the generative model to a large set of data or a time-varying system. 

In the exemplar based dictionary learning methods, a set of training signals $\X \in 
\R^{n \times l}$ is given. Similar to sparse approximations, the $\ell_1$ norm is often preferred as the sparsity-promoting objective \cite{Lewicki00,Lesage05}, and the dictionary learning with this objective can be formulated as follows:
\begin{equation}\label{eq:dl}
\min_\mathsmaller{\V,\D \in \mathcal{D}} \|\V	 \|_{1} + \frac{\lambda}{2} \|\X - \D \V \|_{F}^2 ,
\end{equation}
where $\|\cdot\|_1 = \sum_{i,j} |\{\cdot\}_{i,j}|$ and $\mathcal{D}$ is an admissible set.
Here, $\D \in \R^{n\times p}$ is the dictionary we want to learn and $\V \in \R^{p\times l}$ is the coefficient matrix which is expected to be sparse column-wise.
The reason for $\mathcal{D}$ is to exclude trivial solutions, such as $\X = \D \V$, with $\V \rightarrow \mathbf{0}$ and $\D \rightarrow \infty$. This type of problem is called the \emph{scale ambiguity} of the signal modelling. Various constraints for the resolution of scale ambiguity have been proposed for dictionaries: fixed atom norms \cite{Engan99a,Olshausen97}, fixed Frobenius norm \cite{KreutzDelgado03} and convex balls made using these norms \cite{Lee06,Yaghoobi08a,Mairal10}. Note that the optimisation \eqref{eq:dl} is no longer convex and can be difficult to solve. Different techniques have been proposed to find a sensible solution for (\ref{eq:dl}), where they are often based on alternating optimisations of the objective based upon $\V$ and $\D$. See \cite{Rubinstein10} and \cite{Tosic11} for a review on state of the art methods.

Designing (overcomplete) linear transforms to map some classes of signals to another space with a few significant elements, is not a new subject. It has actually been investigated for almost a decade and many harmonic/wavelet type transforms have been designed with some guarantees on fast decay of coefficients in the transform domain (See for example \cite{Candes05} and \cite{Demanet07}). These transforms are designed such that the perfect reconstructions of the signals are possible, \textit{i.e.} bijective mapping. 

Relatively few research results have been reported in the exemplar based adaptation of analysis operators. This is an important part of data modelling, which has the potential to improve the performance of many current signal processing applications. 

\subsection{Contributions}

This paper investigates the problem of learning an analysis operator from a set of exemplars, for cosparse signal modelling. Our approach to analysis operator learning can be viewed as the counterpart to the optimisation \eqref{eq:dl} in the analysis model setting. Here is the summary of our contributions in this work:

\begin{enumerate}
 \item We propose to formulate \emph{analysis operator learning} (AOL) as a constrained optimisation problem.  Perhaps surprisingly, in contrast to its synthesis equivalent---the dictionary learning problem---the natural formulation of the  analysis operator learning problem as an optimisation problem has trivial solutions even when scaling issues are dealt with. The constraints presented in this report exclude such trivial solutions but are also compatible with a number of conventional analysis operators, e.g. curvelets \cite{Candes05}  and wave-atoms \cite{Demanet07}, giving further support for our proposal.
 \item We provide a practical optimisation approach for AOL problem, based upon projected subgradient method. Clearly, an AOL principle which does not permit computationally feasible algorithm is of limited value. By implementing a practical algorithm, we also open the possibility to cope with large scale problems.
\item We give preliminary theoretical results toward a characterisation of the local optimality conditions in the proposed constrained optimisation problem. This helps us to better understand the nature of the admissible operator set. It also has potential use in the initialisation of our AOL problem.
\end{enumerate}

Our approach is developed from the idea that was primarily suggested in \cite{Yaghoobi11} and \cite{Yaghoobi12}.

\subsection{Related Work}

Remarkably, even though the concept of cosparse analysis modelling has only been very recently pointed out as distinct from the more standard synthesis sparse model \cite{Nam11b,Nam11}, it is already gaining momentum and a few other approaches have already been proposed to learn analysis operators.
The most important challenge in formulation of the AOL as an optimisation problem is to avoid various trivial solutions.
Ophir \textit{et al.} used a random cycling approach to statistically avoiding the optimisation problem becoming trapped in a trivial solution \cite{Ophir11}. Peyr{\'e} \textit{et al.} used a geometric constraint, using a linearisation approach, in the operator update step, to get a sensible \textit{local optimum} for the problem \cite{Peyre11}. In a recent approach, Elad \textit{et al.} have introduced a K-SVD type approach to update each row of the operator at a time \cite{Elad11,Rubinstein12}. 
While these approaches have shown promising empirical results, a specificity of our approach is its explicit expression as an optimisation problem.\footnote{Peyr{\'e} and Fadili \cite{Peyre11} use a similar type of expression to learn an analysis operator in the context of signal denoising, with exemplars of noisy and clean signals.} We expect that this will open the door to a better understanding of the AOL problem, through mathematical characterisations of the optima of the considered cost function.

\subsection{Notation and Terminology}

In this paper we generally use bold letters for vectors and bold capital letters for matrices. We have presented a list of the most frequent parameters in Table \ref{tab:parameters}. We have also presented the corresponding parameters and their range-spaces, used in the related papers, in the same table. The notation $( \: \cdot \: )_{i,j}$, or simply subscript $_{i,j}$, has been used to specify the element locating in the $i^{th}$ row and $j^{th}$ column of the operand matrix. $ \| \cdot \|_1 $, $ \| \cdot \|_2 $ and $ \| \cdot \|_\mathsmaller{F} $ are respectively $\ell_1$, $\ell_2$ and Frobenius norms. With an abuse of notations, we will use $\|\cdot\|_1$ for the norm defined by the entrywise sum of absolute values of the operand matrix, which is different from the $\ell_1$ operator norm for matrices. The notation $\left< \: , \: \right>$ represents the canonical inner-product of two operands for, respectively, $\ell_2$ and Frobenius norms, in the vector-valued and matrix-valued vector spaces. $\tr\{\cdot\}$ denotes the trace of the operand matrix.
The notation $\mathcal{P}$ will be used to represent the orthogonal projection whose range is specified by the subscript, \textit{e.g.} $\mathcal{P}_\mathsmaller{UN}$. Sub- and super- scripts in braces are used to indicate iteration number in the algorithm section. 

The sparsity of $\x$ in $\mathbf{D}$ is noted here by $k$, which is the cardinality of the support of $\v$, where $\v$ is the sparse representation of $\x$. We can similarly define the cosparsity of $\x$ \cite{Nam11}, with respect to $\Omg$, and note it by $q$, which is the cardinality of its co-support, \textit{i.e.} $\operatorname{co-support}(\x) = \{i : 1 \le i \le a, \{\Omg \x\}_i = 0\}$.

Following the notation of \cite{Gribonval10}, a sub-indexed vector or matrix will be determined using a subscript for the original parameter, \textit{e.g.} $\Omg_\mathsmaller{\Lambda}$. Here, $\Omg_\mathsmaller{\Lambda}$ has the dimension of $\Omg$, identical values as $\Omg$ in its support $\Lambda$ and zeros elsewhere. We also use ``bar'', \textit{e.g.} $\bar{\Lambda}$, to denote a complement of the index set $\Lambda$ in this context.

\begin{table}[t]
\caption{The notation of current and their corresponding notations in related papers.}\label{tab:parameters}
\centering
\begin{tabular}{ l  c  c  c  c }
	&	This paper	&	PF\cite{Peyre11}	& \begin{tabular}{c} RFE \cite{Rubinstein12} \\ (OEBP\cite{Ophir11})	\end{tabular} \\
\hline
\\
Observation	&	$\y \in \R^n$	&	$\y \in \R^N$		&	$\y \in \R^d$	\\
Signal		&	$\x \in \R^n$	&	$\x \in \R^N$		& 	$\x \in \R^d$\\
Analysis vector	&	$\z \in \R^a$	&	-			&	-	\\
Dictionary 	&	$\D \in \R^{n \times p}$	&	$\D \in \R^{N \times P}$	&	$\D \in \R^{d \times n}$\\
Synthesis vector	&	$\v \in \R^p$	&	$\mathbf{u} \in \R^P$		&	$\z \in \R^n$\\
Analysis operator	&	$\Omg \in \R^{a \times n}$ &	$\D^*  \in \R^{P \times N}$	&	$\Omg \in \R^{p \times d}$	\\
Training size	&	$l$		&	$n$	&  $R$ ($N$)\\
Cosparsity	&	$q$		&	-	& $l$\\
Sparsity	&	$k$		&	-	& $k$ \\
\end{tabular} 
\end{table}

\subsection{Organisation of the Paper}

In the following section, we formulate the analysis operator learning problem as a constrained optimisation problem. We briefly review some of canonical constraints for the exemplar based learning frameworks and explain why they are not enough. we introduce a new constraint to make the problem `well-posed' in \ref{sec:untf}. 
After the formulation of the AOL optimisation problem, we introduce a practical algorithm to find a sub-optimal solution for the problem in section \ref{sec:algorithm}. Following the algorithm, in Section \ref{sec:simulations}, we will present some simulation results for analysis operator learning with different settings. The simulations are essentially based on two scenarios, synthetic operators and operators for image patches. In the appendix, we will investigate the local optimality conditions of operators for the proposed learning framework. 

\section{Analysis Operator Learning Formulation}\label{sec:AOLF}

The aim of analysis operator learning is to find an operator $\Omg$ adapted to a set of observations of the signals $\y_i := {\mathcal{M}}(\x_i)$, where $\x_i$ is an element (column) of sample data $\X\in\R^{n\times l}$ and ${\mathcal{M}}(\x_i)$ denotes the information available for the learning algorithm. For our purposes, we set $\mathcal{M}(\x_i) := \x_i + \n_i$, where $\n_i \in \R^{m}$ denotes the (potential) noise. The data $\X$ is assumed to be of full rank $n$ since otherwise, we can reduce the dimension of the problem with a suitable orthogonal transform and solve the operator learning in the new low dimension space.

A standard approach for these types of model adaptation problems, is to define a \emph{relevant} optimisation problem such that its optimal solution promotes maximal sparsity of $\mathbf{Z} := \Omg \mathbf{X}$. The penalty $\norm{\cdot}_p := \sum |\cdot|^p, \: 0 < p \le 1$ can be used as the sparsity promoting cost function. As hinted in the introduction, we will use the convex $\ell_1$ penalty. The extension to an $\ell_p, \: p<1 ,$ AOL is left for a future work.

\subsection{Constrained Analysis Operator Learning (CAOL)}

Momentarily assume that we know $\X$. Unconstrained minimisation of $\|\Omg \X\|_1$, based on $\Omg$, has some trivial solutions. A solution for such a minimisation problem  is $\Omg = \mathbf{0}$! A suggestion to exclude such trivial solutions is to restrict the solution set to an admissible set $\mathcal{C}$. Not assuming that $\X$ is known any more, AOL can thus be formulated as,

\begin{equation}\label{eq:l1stl2C}
 \min_{\Omg,\X} \|\Omg \mathbf{\X}\|_1,\ \subto \ \Omg \in \mathcal{C}, \ \norm{\Y - \X}_F \le \sigma
\end{equation}
where $\sigma$ is the parameter corresponding to the noise. We call the AOL to be a noiseless operator learning when $\sigma = 0$. 

We prefer to use an alternative, regularised version of (\ref{eq:l1stl2C}), using a Lagrangian multiplier $\lambda$, to simplify its optimisation. The reformulated AOL is the following problem,

\begin{equation}\label{eq:l1l2stC}
 \min_{\Omg,\X} \|\Omg \mathbf{X}\|_1 + \frac{\lambda}{2} \|\Y - \X \|_{F}^2,\ \subto \ \Omg \in \mathcal{C}
\end{equation}

If $\lambda \rightarrow \infty$, problem (\ref{eq:l1l2stC}) is similar to the noiseless case.
In the following section, we explore some candidates for $\mathcal{C}$.

\begin{figure}[t]
\centering
\centerline{\epsfig{figure=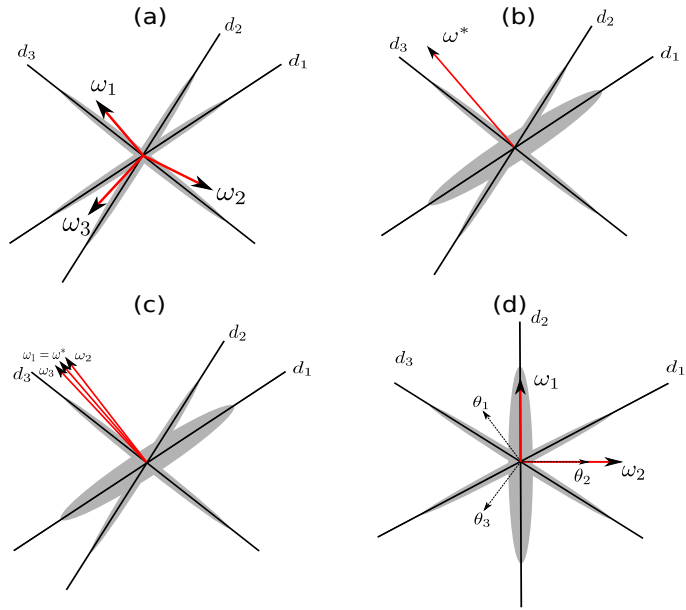,width=9cm, height=8cm}}
\caption{Data clouds around some union of subspaces and possible analysis operators: a) an ideal operator, b) the optimal (rank-one) operator using a row norm constraint, c) the optimal operator using the full-rank, row norm constraint and d) the optimal operator using tight frame constraint.}
\label{fig:failure}
\end{figure}
%
%
%

\subsection{Constraints for the Analysis Operators}

A suitable admissible set should exclude undesired solutions of AOL while minimally affecting the rest of the search space. We name $\Omg = \mathbf{0}$ and $\operatorname{Rank}\{\Omg\} = 1$ as some undesired solutions of (\ref{eq:l1stl2C}). These operators clearly do not reveal low-dimensional structures of the signals of interest. For simplicity, we again assume that $\X$ is given, \textit{i.e} $\sigma = 0$. 

Here, we intially propose some constraints for the problem (\ref{eq:l1stl2C}) and explain why some of them can not individually exclude undesired solutions. A combined constraint $\mathcal{C}$, which is the Uniform Normalised Tight Frame (UNTF), will be introduced subsequently. The proposed constraint is smooth and differentiable.

\subsubsection{Row norm constraints are insufficient}
The first constraint is on the norms of rows of $\Omg$, \textit{i.e.} $\|\omg_i\| = c$ for the $i^{th}$ row. We can find the solution of the optimisation by first finding $\omg^* \in \mathbb{R}^n$ that minimises $\|\omg^T \X\|_1$. Then, the optimum is obtained by repeating $\omg^*$, i.e., $\Omg_{1}^* := [\omg_i = \omg]_{i \in [1,a]}^\mathsmaller{T}$ is the minimum. Of course, $\operatorname{Rank} \{ \Omg_{1}^* \} = 1$ (see Figure \ref{fig:failure}.b), and the solution is not interesting enough.


\subsubsection{Row norm + full rank constraints are insufficient}


The set $\mathcal{C}_F$ of full rank operators is not closed, and the operator $\Omg_{1}^*$ belongs to its closure. As a result, the problem (\ref{eq:l1stl2C}) with $\sigma = 0$ does not admit a minimiser but there is a sequence of operators, arbitrarily close to $\Omg_{1}^*$, that yield an objective function arbitrarily close to its infimum value.

\subsubsection{Tight frame constraints are insufficient}
In a complete setting $a = n$, an orthonormality constraint can resolve the ill-conditioning of the problem. The rows of $\Omg$ are geometrically as separated as possible. Letting $a > n$, the orthonormality constraint is not further applicable, \textit{i.e.} more rows than the dimension of space. In this case, the orthonormality constraint in the ambient space, $\forall i \ne j, \ \omg^i \bot \omg^j$ and $\|\omg^i\|_2  = \|\omg^j\|_2  = 1$, where $\omg^i \operatorname{and} \omg^j \in \mathbb{R}^{a}$ are respectively the $i^{th}$ and $j^{th}$ \emph{columns} of $\Omg$, is possible. The admissible set of this constraint is the set of tight frames in $\mathbb{R}^{a \times n}$, \textit{i.e.} $\Omg^T \Omg = \mathbf{I}$, where $\mathbf{I}$ is the identity operator in $\mathbb{R}^{n}$. The admissible set $\mathcal{C} = \{\Omg \in \mathbb{R}^{a \times n} : \Omg^T \Omg = \mathbf{I}\}$ is smooth and differentiable, which is a useful property for optimisation over this set.
Such a constraint actually constructs a manifold in the space of $\mathbb{R}^{a \times n}$, called the (orthogonal) Stiefel manifold $\operatorname{St}(a,n)$ \cite{Absil08}. 


Although the tight frame constraint may look appropriate to avoid ``trivial'' solutions to (\ref{eq:l1stl2C}), preliminary empirical and theoretical investigations indicate that the analysis operator minimising (\ref{eq:l1stl2C}), using this constraint, is always an orthonormal basis completed by zero columns (see Figure \ref{fig:failure}.d). This is possibly caused by the fact that the zero elements in such operators, are orthogonal to \textbf{all} finite signals, without giving any insight about their directions, and the operators can thus be truncated to produce complete operators. Therefore, this constraint does not bring anything new compared to the orthonormality constraint.

\subsubsection{Proposed constraint: Uniform Normalised Tight Frame} \label{sec:untf}
This motivates us to apply an extra constraint. Here, we choose to combine the uniformly normalised rows and the tight frame constraints, yielding the UNTF constraint set. 
This constraint will be used in the rest of paper and the analysis of Appendix \ref{app:theory} is based on this admissible set.

The UNTF's are actually in the intersection of two manifolds, uniform normalised (UN) frames manifold and tight frames (TF) manifold. There exist some results that guarantee the existence of such tight frames, \textit{i.e.} non-emptiness of the intersection of two manifolds, for any arbitrary dimensions $n$ and $a$ \cite{Benedetto03}. This is indeed an important fact, as the optimisation with such a constraint is then \textit{always} feasible.

This constraint can be written as follows,

\begin{equation}\label{eq:C}
 \mathcal{C} = \{\Omg \in \R^{a \times n} \ : \ \Omg^\mathsmaller{T}\Omg = \mathbf{I},\ \forall i \ \|\omg_i\|_2 = c \},
\end{equation} 
and $\omg_i$ is the $i$th row of $\Omg$. Since $\mathcal{C}$ is not convex, (\ref{eq:l1l2stC}) may have many separate \textit{local optima}. Finding a global optima of (\ref{eq:l1l2stC}) is not easy and we often find a local optima using variational methods. Using variational analysis techniques, we can also check if an operator $\Omg$ is a local optimum. Details are presented in Appendix \ref{app:theory}. 

The exact recovery, using a variational analysis, in this setting, is essentially based on the following property,

\begin{definition}[Locally Identifiable]\label{def:localident}
 An operator $\Omg$ is locally identifiable from some (possibly noisy) training cosparse signals $\Y$, if it is a local minimum of an optimisation problem with a continuous objective and locally connected constraint.
\end{definition}
This definition of local identifiability is a generalisation of the (global) identifiability as here we can not find the global solution. Such a property has been previously investigated in the context of dictionary learning in \cite{Gribonval10} and \cite{Geng11}. 

\subsection{Extension of UNTF Constraints}\label{sec:extraconstraint}
In some situations, we may not necessarily want to require that $\Omg$ form a frame for the whole of $\R^n$. The inspiration for such a case comes from the finite difference operators which is closely tied to the popular TV-norm. When we model cartoon-like piece-wise constant images, an analysis operator $\Omg$ need only to detect transitions in neighboring pixel values. Therefore, all the rows of $\Omg$ can be orthogonal to the DC component.

We can facilitate learning similar rank-deficient operators by extending the UNTF constraints.
Here, we know that $\Omg \perp \mathcal{N}$, where $\mathcal{N}$ is the known null-space. We can modify (\ref{eq:C}) to construct the following constraint,

\begin{equation}\label{eq:CNull}
 \mathcal{C}_\mathcal{N} = \{\Omg \in \R^{a \times n} \ : \ \Omg^\mathsmaller{T}\Omg = \mathcal{P}_\mathsmaller{\mathcal{N}^\perp},\ \forall i \ \|\omg_i\|_2 = c \},
\end{equation} 
where $\mathcal{P}_\mathsmaller{\mathcal{N}^\perp}$ is the orthogonal projection operator onto the span of $\mathcal{N}^\perp$.

In the following, we introduce a practical algorithm to find a suboptimal solution of (\ref{eq:l1l2stC}), constrained to $\mathcal{C}$ or $\mathcal{C}_\mathcal{N}$, using a projected subgradient based algorithm.

\section{Constrained Analysis Operator Learning Algorithm}\label{sec:algorithm}

Cosparse signals, by definition, have a  large number of zero elements in the analysis space. In this framework, a convex formulation for recovering a signal $\x$ from its (noisy) observation is to minimise the following convex program,

\begin{equation} \label{eq:nl1}
\min_{\x} \|\Omg \x\|_1 + \frac{\lambda}{2}  \|\y - \x\|_{2}^2,
\end{equation}
where $\Omg$ is a fixed analysis operator. As (\ref{eq:nl1}) is a convex problem, many algorithms have been proposed to solve this program very efficiently (see \cite{Afonso09},\cite{Becker09} and \cite{Becker11} as some examples of the latest methods). These methods can easily be extended to the matrix form to find $\X$ given $\Y$. Unfortunately, the AOL problem \eqref{eq:l1l2stC}, due to the freedom to also change $\Omg \in \mathcal{C}$ in this formulation, is very difficult.


This difficulty is similar to what we face in the \emph{dictionary learning} for synthesis sparse approximation. In the dictionary learning problem, we have also a joint optimisation problem (\ref{eq:dl}) that, in a simple setting, can be minimised in an alternating optimisation approach \cite{Yaghoobi09} as follows,


\begin{algorithmic}[0]
\STATE \textbf{\textit{Dictionary Learning:}}
\STATE \textbf{initialisation:} $\D^{[0]}$, $\V^{[0]} = \mathbf{0}$, $i = 0$,
\WHILE{not converged}
\STATE $\V^{[i+1]} = \argmin_{\V} \| \V \|_1  + \frac{\lambda}{2} \|\X - \D^{[i]} \V \|_{F}^2$,
\STATE $\D^{[i+1]} = \argmin_{\D \in \mathcal{D}} \| \X - \D \V^{[i+1]}\|_{F}^2$
\STATE $i = i+1$
 \ENDWHILE .
\end{algorithmic}


An alternating minimisation technique can also be used for the AOL problem to seek a local minimum. In this framework we first keep $\X$ fixed and update the operator $\Omg$. We then update the cosparse approximations of the signals while keeping the operator fixed. Such an alternating update continues while the new pair $(\X, \Omg)$ is very similar to the previous  $(\X, \Omg)$ or is repeated for a predetermined number of iterations. A pseudocode for such an alternating minimisation AOL is presented in Algorithm \ref{alg:aola}, Analysis Operator Learning Algorithm (AOLA). Each subproblem, \textit{i.e} line \ref{eq:OmgUpdate} or line \ref{eq:XUpdate}, of AOLA can be solved separately based upon a single parameter. Although the problem in line \ref{eq:XUpdate} is convex and it can thus be solved in a polynomial time, the problem in line \ref{eq:OmgUpdate} is \emph{not} a convex problem due to the UNTF constraint $\mathcal{C}$.

Note that when the cosparse matrix $\X$ is given, \textit{i.e.} in a noiseless scenario $\sigma = 0$, the algorithm only has the operator update step, which we need to repeat until convergence. This step is also called noiseless AOL\cite{Yaghoobi11}.

\begin{algorithm}[t]
 \caption{Analysis Operator Learning Algorithm (AOLA)} \label{alg:aola}
\begin{algorithmic}[1]
\STATE \textbf{initialisation:} $\Omg^{[0]}$, $\X^{[0]} = \Y$, $i = 0$,
\WHILE{not converged}
\STATE $\Omg^{[i+1]} = \argmin_{\Omg \in \mathcal{C}} \|\Omg \X^{[i]}\|_1$, \label{eq:OmgUpdate}
\STATE $\X^{[i+1]} = \argmin_{\X} \|\Omg^{[i+1]}\X\|_1 + \frac{\lambda}{2} \|\Y - \X \|_{F}^2$ \label{eq:XUpdate}
\STATE $i = i+1$
 \ENDWHILE .
\end{algorithmic}
\end{algorithm}

\subsection{Analysis Operator Update} \label{sec:AOU}

We use a projected subgradient type algorithm to solve the operator update subproblem in line \ref{eq:OmgUpdate} of the AOLA. The subgradient of the objective is $\partial f(\Omg) = \esgn(\Omg \X^{[i]}) \mathbf{X}^{{[i]}^\mathsmaller{T}}$, where $\esgn$ is an extended set-valued sign function defined as follows,

\begin{equation*}
 \begin{split}
 \{\esgn(\mathbf{A})\}_\mathsmaller{i\,j} =& \esgn(\mathbf{A}_{i\,j}) \\ 	
 \esgn(\mathbf{a}) =& \begin{cases}
	                       	1 & \mathbf{a} > 0,\\
				[-1, 1] & \mathbf{a} = 0,\\
				-1 & \mathbf{a} < 0.
	                      \end{cases}
\end{split} 
\end{equation*}
In the projected subgradient methods, we have to choose a value in the set of subgradients. We randomly choose a value in $[-1,\, 1]$, when the corresponding element is zero.

After the subgradient descent step, the modified analysis operator is no longer UNTF and needs to be projected onto the UNTF set. Unfortunately, to the authors' knowledge, there is no \emph{analytical} method to find the projection of a point onto this set. Many attempts have been done to find such an (approximate) projection, see for example \cite{Tropp05} for an alternating projections and \cite{Bodmann10} for an ordinary differential equations based method. We rely on an alternating projections approach.
 
Projection of an operator with non-zero rows, onto the space of fixed row norm frames is easy and can be done by scaling each row to have norm $c$. We use $\mathcal{P}_{UN}$ to denote this projection.
If a row is zero, we set the row to a normalised random vector. This means that $\mathcal{P}_{UN}$ is not uniquely defined. This is due to the fact that the set of uniformly normalised vectors is not convex.
The projection can be found by,
\begin{equation*}
 \mathcal{P}_{UN}\{\Omg\} = [\mathcal{P}_{UN} \{\omg_i\}]_i ,
\end{equation*}
\begin{equation*}
 \mathcal{P}_{UN} \{\omg\}:= \begin{cases}
                              		   \frac{\omg}{\|\omg\|_\mathsmaller{2}} & \|\omg\|_\mathsmaller{2} \ne 0\\
					   \nu & \operatorname{otherwise},\\
                    			\end{cases}
\end{equation*}
where $\nu$ is a random vector on the unit sphere. 

Projection of a full rank matrix onto the tight frame manifold is possible by calculating a singular value decomposition of the linear operator \cite{Tropp05}. Let $\Omg \in \mathbb{R}^\mathsmaller{a \times n}$ be the given point and $\Omg = \mathbf{U} \: \Sigma \: \mathbf{V}^\mathsmaller{T}$ be a singular value decomposition of $\Omg$ and $\mathbf{I}_\mathsmaller{a \times n}$ be a diagonal matrix with identity on the main diagonal. The projection of $\Omg$ can be found using,
\begin{equation*}
\mathcal{P}_{TF}\{\Omg\} =  \mathbf{U} \: \mathbf{I}_\mathsmaller{a \times n} \: \mathbf{V}^\mathsmaller{T}.
\end{equation*}

Note that, although there is no guarantee to converge to a UNTF using this method, this technique practically works very well \cite{Tropp05}. As the projected subgradient continuously changes the current point, which needs to be projected onto UNTF's, we only use a single pair of projections at each iteration of the algorithm, to reduce the complexity of the algorithm, where the algorithm now asymptotically converges to a UNTF fixed point.
To guarantee a uniform reduction of the objective $f(\Omg)$, we can use a simple line search technique to adaptively reduce the stepsize. It indeed prevents any large update, which increases $f$\footnote{The stability of algorithm is guaranteed in a Lyapunov sense, \textit{i.e.} the operator $\Omg_\mathsmaller{[k]}$ is enforced to remain bounded, when $k \rightarrow \infty$.} . A pseudocode of this algorithm is presented in Algorithm \ref{alg:pcdl}, where $K_\mathsmaller{max}$ is the maximum number of iterations.

For the constraint set $\mathcal{C}_\mathcal{N}$ mentioned in Section \ref{sec:extraconstraint},
we present a simple modification to the previous algorithm in order to constrain the operator to have a specific null-space. 
Here, we need to find a method to (approximately) project an operator onto $\mathcal{C}_\mathcal{N}$. Following the alternating projection technique, which was used earlier to project onto $\mathcal{C}$, we only need to find the projection onto the set $\{\Omg \in \R^{a \times n} \ : \ \Omg^\mathsmaller{T}\Omg = \mathcal{P}_\mathsmaller{\mathcal{N}^\perp}\}$, as we already know how to project onto UN, \textit{i.e.} $\mathcal{P}_\mathsmaller{UN}$. To project an $\Omg$ onto the set of TF's in $\mathcal{N}^\perp$, we need to compute the singular value decomposition of $\Omg$, projected into the orthogonal complement space of 
$\mathcal{N}$. The projection onto $\mathcal{N}^\perp$, $\mathcal{P}_\mathsmaller{\mathcal{N}^\perp}$, can be found using an arbitrary orthonormal basis for $\mathcal{N}$ and simply subtracting the projection onto this basis, from the original $\Omg$. We can then rewrite the decomposition as $\mathcal{P}_\mathsmaller{\mathcal{N}^\perp}\{ \Omg\} = \U \: \Sigma \: \V^\mathsmaller{T}$. 
If the dimension of $\mathcal{N}$ is $r$, only $n-r$ singular values of $\mathcal{P}_\mathsmaller{\mathcal{N}^\perp}\{ \Omg\}$ can be non-zero. We find the constrained projection as follows,
\begin{equation*}
\mathcal{P}_{TF{\perp}\mathcal{N}}\{\mathbf{\Omg}\} =  \mathbf{U} \: \mathbf{Ir}_\mathsmaller{a \times n} \: \mathbf{V}^\mathsmaller{T},
\end{equation*}
where matrix $\mathbf{Ir}_\mathsmaller{a \times n}$ is a diagonal matrix in $\R^{a \times n}$, with identity values on the first $n-r$ diagonal positions, while the last $r$ elements are set to zero.

We therefore need to alternatingly use $\mathcal{P}_{TF{\perp}\mathcal{N}}$ and $\mathcal{P}_{UN}$ to find a point in the intersection of these two constraint sets. In Algorithm \ref{alg:pcdl}, we only need to replace $\mathcal{P}_{TF}$ with $\mathcal{P}_{TF{\perp}\mathcal{N}}$ to consider the new constraint $\mathcal{C}_\mathcal{N}$.

\begin{algorithm}[t]
 \caption{Projected Subgradient Based Analysis Operator Update} \label{alg:pcdl}
\textbf{Input:} $\X^{[i]}$, $K_\mathsmaller{\max}$, $\Omg^{[i]}$, $\eta$, $\epsilon \ll 1$, $\rho < 1$,
\begin{algorithmic}[0]
 \STATE \textbf{initialisation:} $k = 1$,  $\Omg_\mathsmaller{[0]} = \mathbf{0}$, $\Omg_\mathsmaller{[1]} = \Omg^{[i]}$
\WHILE{$\epsilon \le \|\Omg_\mathsmaller{[k]} - \Omg_\mathsmaller{[k-1]}\|_\mathsmaller{F} $ and $k \le K_{\max}$}
\STATE $\Omg_\mathsmaller{G} \in \partial f(\Omg_\mathsmaller{[k]}) = \esgn(\Omg_\mathsmaller{[k]}\X^{[i]}) \mathbf{X}^{{[i]}^\mathsmaller{T}}$
\STATE $\Omg_\mathsmaller{[k+1]} = \mathcal{P}_{UN}\left\{ \mathcal{P}_{TF}\left\{\Omg_\mathsmaller{[k]} - \eta \ \Omg_\mathsmaller{G} \right\}\right\}$
\WHILE{$f(\Omg_\mathsmaller{[k+1]}) > f(\Omg_\mathsmaller{[k]})$}
\STATE $\eta = \rho . \eta$
\STATE $\Omg_\mathsmaller{[k+1]} = \mathcal{P}_{UN}\left\{ \mathcal{P}_{TF}\left\{\Omg_\mathsmaller{[k]} - \eta \ \Omg_\mathsmaller{G} \right\}\right\}$
 \ENDWHILE
\STATE $k = k+1$
 \ENDWHILE
\STATE \textbf{output:} $\Omg^{[i+1]} = \Omg_\mathsmaller{[k-1]}$
\end{algorithmic}
\end{algorithm}

\subsection{Cosparse Signal Update} \label{sec:CSU}

When $\Omg$ is known or given by line \ref{eq:OmgUpdate} of AOLA, a convex program based on $\X$ needs to be solved. Although this program is a matrix valued optimisation problem, it can easily be decoupled to some vector valued subproblems based upon the columns of $\X$. One approach is to individually solve each subproblem. Here we present an efficient approach to solve this program in the original matrix-valued form, despite the challenging nature of it, due to its large size.

The main difficulties are, a) $\ell_1$ is not differentiable and b) $\mathbf{\Omega}$ inside the $\ell_1$ penalty, does not allow us to use conventional methods for solving similar $\ell_1$ penalised problems. We here use the Douglas-Rachford Splitting (DRS) technique to efficiently solve the cosparse signal approximation of the AOLA. It is also called the alternating direction method of multipliers (ADMM) in this setting, see \cite{Eckstein92} for a brief overview. This technique has indeed been used
for the Total Variation (TV) and analysis sparse approximations in \cite{Afonso09}\cite{Goldstein09}\footnote{\cite{Goldstein09} derives the formulation by incorporating the Bregman distance. However, it has been shown that the new method, called Alternating Bregman Splitting method, is the same as DRS, applied to the dual problems \cite{Setzer09}.}. We here only present a simple version of the DRS technique, carefully tailored for this problem. 
The problem is a constrained convex program with two parameters $\Z  = \Omg^{[i+1]} \X$ and $\X$ as follows,

\begin{equation}\label{eq:split}
\min_{\X,\Z} {\|\Z\|_1 + \frac{\lambda}{2}  \|\Y - \X\|_{F}^2} \ \subto \ \Z = \Omg^{[i+1]} \X.
\end{equation}
 
The Augmented Lagrangian (AL) \cite{Rockafellar74} method is applied to solve (\ref{eq:split}). In the Lagrangian multiplier method we use the dual parameter $\B \in \R^{a \times l}$ and add a penalty term, $\left< \B ,\Omg^{[i+1]} \X - \Z \right>$. In the AL, we also add an extra quadratic penalty related to the constraint and derive the new objective $g(\X,\Z,\B)$ as follows, 
\begin{equation*}
\begin{split}
g(\X,\Z,\B) =& \|\Z\|_1 + \frac{\lambda}{2}  \|\Y - \X\|_{F}^2 + \gamma \left< \B , \Omg^{[i+1]} \X - \Z \right> \\ 
& \quad+ \frac{\gamma}{2} \|\Omg^{[i+1]} \X - \Z \|_{F}^2 \\
=& \|\Z\|_1 + \frac{\lambda}{2}  \|\Y - \X\|_{F}^2 + \frac{\gamma}{2} \|\B + \Omg^{[i+1]} \X - \Z\|_{F}^2 
\\& \quad- \frac{\gamma}{2}\|\B\|_{F}^2 ,
\end{split}
\end{equation*}
where $0 < \gamma \in \R^+$ is a constant parameter. According to the duality property, the solution of $\max_\mathsmaller{\B} \min_{\mathsmaller{\X},\mathsmaller{\Z}}
g(\X,\Z,\B)$ coincides with the solution of (\ref{eq:split}). Using the DRS method, we iteratively optimise a convex/concave surrogate objective $g_s(\X,\Z,\B,\B_{[k]}) =
g(\X,\Z,\B) - \|\B - \B_{[k]}\|_{F}^2$, where $\B_{[k]}$ is the current estimation of $\B$. The fixed points of the iterative updates of $g_s(\X,\Z,\B,\B_{[k]})$
are the same as $g(\X,\Z,\B)$, as the extra term $\|\B - \B_{[k]}\|_{F}^2$ vanishes in any fixed points. $g_s(\X,\Z,\B,\B_{[k]})$ is convex with respect to $\Z$ and $\X$ and concave with respect to $\B$. We can thus iteratively update each of the parameters, while keeping the rest fixed, see Algorithm \ref{alg:csua}.
In this algorithm, $\mathcal{S}_\alpha$, with an $\alpha > 0$ , is the entrywise
soft-threshold operator defined by $\mathcal{S}_\alpha(\beta) = \beta - \alpha\sgn(\beta) \operatorname{if} |\beta| \ge \alpha \operatorname{and} \; 0 \: \operatorname{otherwise}$ \cite{Donoho94}. Note that the update formula for $\X$ in Algorithm \ref{alg:csua}, line \ref{eq:XinUpdate}, involves a matrix inversion, which is computationally expensive. As $\mathbf{\Omega}$ is a tight frame here, the matrix inversion is significantly simplified using the fact that $\Omg^\mathsmaller{T}\Omg$ is identity. In this case, the operator $(\lambda \mathbf{I} + \gamma \ \mathbf{\Omega}^\mathsmaller{T}\mathbf{\Omega})^{-1}$ is simply a scaling with $\frac{1}{\lambda + \gamma}$.

We iterate Algorithm \ref{alg:csua} for a number of iterations $K_\mathsmaller{max}$ or until the
parameters cease to change significantly. Although the convergence of
the iterative updates of this algorithm can separately be investigated,
it can also be deduced using the fact that it is a particular case of
DRS, which converges under mild conditions \cite{Eckstein92}. 

\begin{algorithm}[t]
 \caption{DRS Based Cosparse Signal Update} \label{alg:csua}
\textbf{Input:} $i$, $\Omg^{[i+1]}$, $K_{\max}$, $\X^{[i]}$, $\gamma$, $\lambda$, $\epsilon \ll 1$, $\lambda,\gamma$
\begin{algorithmic}[1] 
 \STATE \textbf{initialisation:} $k = 1$,  $\X_{[k]} = \X^{[i]}$, $\Omg = \Omg^{[i+1]}$, $\B_{[k]} = \mathbf{0}$, \\$\Z_{[k]} = \Omg \X_{[k]}$
\WHILE{$\epsilon \le \|\X_{[k]} - \X_{[k-1]}\|_\mathsmaller{F} $ and $k \le K_{\max}$}
\STATE ${\X}_{[k+1]} = (\lambda \mathbf{I} + \gamma \Omg^\mathsmaller{T}\Omg)^\mathsmaller{-1} (\lambda  \Y + \gamma \Omg^\mathsmaller{T}  (\Z_{[k]} -  \B_{[k]}))$ \label{eq:XinUpdate}
\STATE $\Z_{[k+1]} = \mathcal{S}_\mathsmaller{\frac{1}{\gamma}} \left\{\Omg \X_{[k+1]} + \B_{[k]} \right\}$
\STATE $\B_{[k+1]} = \B^{[k]}  + (\Omg \X_{[k+1]} - \Z_{[k+1]})$
\STATE $k = k+1$
 \ENDWHILE
\STATE \textbf{output:} $\X^{[i+1]} = \X_{[k-1]}$
\end{algorithmic}
\end{algorithm}


\section{Simulations}\label{sec:simulations}

\begin{figure}[t]
\centering
\centerline{\epsfig{figure=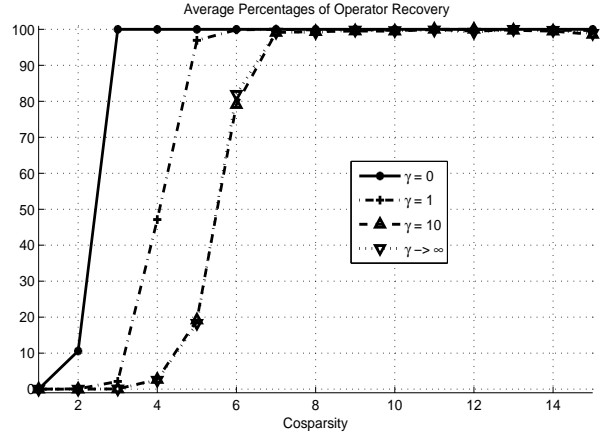,width=9.5cm, height=6cm}}
\vspace{-3mm}
\caption{The average percentage of operator recovery for different $\gamma$'s, where $\gamma$ controls how far is the starting point $\Omg_\mathsmaller{in}$ form $\Omg_\mathsmaller{0}$. The $x$-axis presents the cosparsity of the synthetic data.}
\label{recLambda}
\end{figure}
%

We present two categories of simulation results in this section. The first part is concerned with empirical demonstration of (almost) exact recovery of the reference analysis operators. An astute reader might ask: why should we care about exact recovery at all? Is it not enough to show that learned operators are good? This is, of course, true. However, when the signals are generated to be cosparse with respect to a reference operator or they are known to be cosparse with respect to a known analysis operator, the recovery of those operators is a good demonstration of the effectiveness of our method. Going further, one may imagine situations where the rows of the reference analysis operators explain certain properties/dynamics of the signals; for example, finite differences for piecewise constant images detect edges.
Obviously, the recovery of reference operators in these contexts is significant.

In the second part, we are interested in demonstrating the effectiveness of learned operators in a practical task, namely, image denoising. 

\subsection{Exact recovery of analysis operators}
For the first experiment, the reference operators and the cosparse signal data set were generated as follows:
A random operator $\Omg_{\mathsmaller{0^-}} \in \mathbb{R}^\mathsmaller{24 \times 16}$ was generated using \textit{i.i.d.} zero mean, unit variance normal random variables\footnote{$\Omg_{\mathsmaller{0^-}}$ is not necessarily a UNTF and needs to be projected onto the set of UNTF's.}. The reference analysis operator $\Omg_\mathsmaller{0}$ is made by alternatingly projecting $\Omg_\mathsmaller{0^-}$ onto the sets of $UN$'s and $TF$'s. A set of training samples was generated, with different cosparsities, by randomly selecting a normal vector in the orthogonal complement space of a randomly selected $q$ rows of $\Omg_\mathsmaller{0}$. Such a vector $\mathbf{y}_i$ has (at least) $q$ zero components in $\Omg_\mathsmaller{0} \mathbf{y}_i$, and it is, thus, $q$ cosparse. 

To initialise the proposed algorithm, we used a linear model to generate the initial $\Omg$ by combining the reference operator $\Omg_\mathsmaller{0}$ and a normalised random matrix $\mathbf{N}$, \textit{i.e.} $\Omg_{\mathsmaller{in}} = \Omg_\mathsmaller{0} + \gamma \, \mathbf{N}$, and then alternatingly projecting onto $UN$ and $TF$. It is clear that when $\gamma$ is zero, we actually initialise $\Omg$ with the reference model $\Omg_\mathsmaller{0}$ and when $\gamma \rightarrow \infty$, the initial $\Omg_\mathsmaller{in}$ will be random. 


First, we chose a set of size $l = 768$ of such training corpus and used the noiseless formulation (\ref{eq:l1stl2C}), where $\sigma = 0$. The AOL algorithm in this setting is just the projected subgradient base algorithm, Algorithm \ref{alg:pcdl}, which was iterated $50000$ times. To check the local optimality of the operator and the size of basin of attraction, we chose $\gamma = 0,\, 1, 10$ and $\infty$. The average percentage of operator (rows) recovery, \textit{i.e.} the maximum $\ell_2$ distance of any recovered row and the closest row of the reference operator, is not more than $\sqrt{.001}$, for different cosparsity and 100 trials, are plotted in Figure \ref{recLambda}. We practically observe that the operator is the local optimum even when the cosparsity of the signal is as low as 3. We also see that the average recovery reduces by starting from a point far from the the actual reference operator or a random operator, which is the case $\gamma \rightarrow \infty$.

\begin{figure}[t]
\vspace{-2mm}
\centering
\centerline{\epsfig{figure=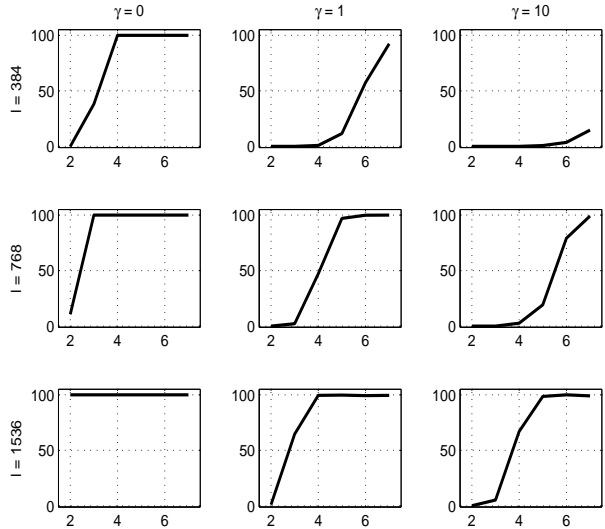,width=9.5cm, height=8cm}}
\vspace{-5mm}
\caption{The average percentage of operator recovery with different training set population size $l$. The $x$-axis presents the cosparsity of the signals. }
\label{recL}
\end{figure}

\begin{figure}[t]
\vspace{-2mm}
\centering
\centerline{\epsfig{figure=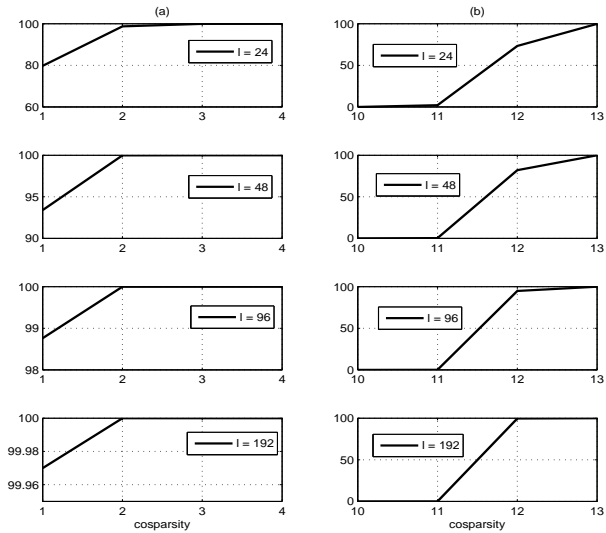,width=9.5cm, height=8cm}}
\vspace{-5mm}
\caption{The local-identifiability check by randomly generating vectors in the admissible set of Lemma \ref{lem:iffnoiseless} of Appendix \ref{app:theory} and checking the credibility of the inequality in (a) Lemma \ref{lem:iffnoiseless}- Equation (\ref{eq:iffnoiseless}); and (b) Theorem \ref{theo:identifiability}- Equation (\ref{eq:identifiability}). }
\label{identifiability}
\end{figure}

We now investigate the role of $l$ on the average operator recovery by some simulations. We kept the previous experiment settings and repeated the simulations for two new training sets, with populations of $l = 384$ and $1536$, which are $1/2$ and $2$ times of the population in the previous experiments. We show the average operator recovery for $\gamma = 0, \, 1, \, 10$ in Figure \ref{recL}. The simulation results show not only that $\Omg_\mathsmaller{0}$ can be locally identified with even less cosparse signals, smaller $q$, but the basin of attraction is also extended and now the reference operator can be recovered by starting from a distant initial point, even using 2 cosparse signals.

In the next two experiments, we show that if the cosparsity is low, \textit{i.e.} small $q$, then the analysis operator cannot be ``recovered'': it is not a solution of the proposed optimization problem (\ref{eq:l1stl2C}). One way to demonstrate it would consist in solving the large scale convex optimisation program (\ref{eq:noiselesslinerized})
expressed in Proposition \ref{pro:localiden} of the Appendix \ref{app:theory}.
Alternatively, we can use Lemma \ref{lem:iffnoiseless} (respectively Theorem \ref{theo:identifiability}) of the appendix, to possibly find a matrix $\Delta_z$, which violates the conditions. Here, we randomly generated 1000 $\Delta_z$ in $\mathcal{T}_{\mathcal{C}_s}$, and checked whether the inequality (\ref{eq:iffnoiseless}) (respectively (\ref{eq:identifiability})) was satisfied. We repeated this process for 10 different pseudo-randomly generated $\Omg$ and plotted the percentages of $\Delta_z$ satisfying the inequality in Figure \ref{identifiability}, respectively in the left column (a) and the right column (b). We have also mentioned the training size, \textit{i.e.} $l$ on each row of this figure. From left column, we can see that with 1-cosparse signals, there are operators which are not local minimum of the program (\ref{eq:l1stl2C}) ($\sigma = 0$). We can also see that the relaxation used to derive Theorem \ref{theo:identifiability}, makes the result of this Theorem very conservative, \textit{i.e.} there are many cases which do not satisfy this theorem, but they still can be locally-identifiable.

%

\begin{figure}[t]
\centering
\epsfig{figure=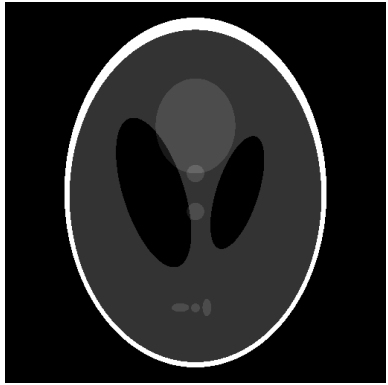, width=6.5cm }
\vspace{-5mm}
\caption{A 512 by 512 Shepp-Logan phantom image which was used as the training image. }
\label{sheppLogganPhantom}
\end{figure}

\begin{figure}[t]
\centering
\vspace{-2mm}
\centerline{\epsfig{figure=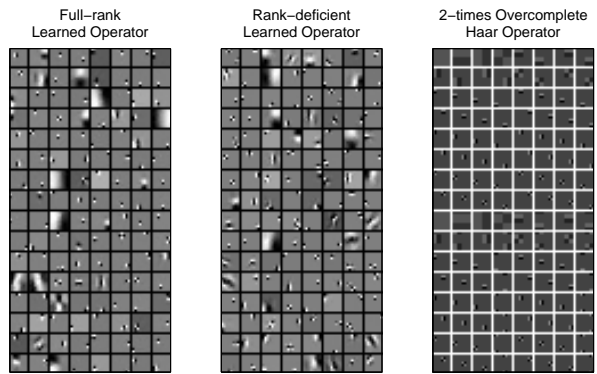,width=10cm}}
\vspace{-5mm}
\caption{The rows of the full-rank learned operator (left panel), the rank-deficient learned operator (middle panel) and a two-times oevrcomplete Haar operator (right panel), in the form of 8 by 8 blocks.}
\label{LearnedOperatorRows}
\end{figure}

\begin{figure}[t]
\centering
\vspace{-2mm}
\centerline{\epsfig{figure=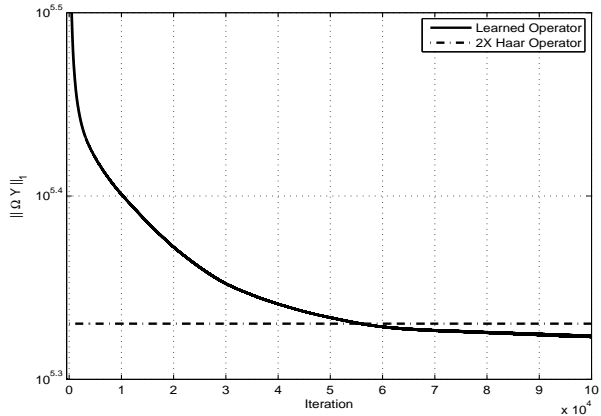,width=9.0cm, height=6cm}}
\vspace{-5mm}
\caption{The value of objective in (\ref{eq:noislessoptformulation}), in comparison with a base line, \textit{i.e.} the objective value with a two times overcomplete Haar wavelet operator.}
\label{costComp}
\end{figure}



For the last experiment, we chose the Shepp Logan Phantom, a well-known piecewise constant image in the Magnetic Resonance Imaging (MRI) community (see Figure \ref{sheppLogganPhantom}). We used $l = 16384$ blocks of size 8 by 8, randomly selected from the image, such that, except one DC training sample, all the training images contain some edges. We learned a $128 \times 64$ operator for the vectorised training image patches, by initialising the (noise-less) AOL algorithm with a pseudo-random UNTF operator. We used two different constraints $\mathcal{C}$ and $\mathcal{C}_\mathsmaller{\mathcal{N}}$ for the learning operator $\Omg$, \textit{i.e.} (\ref{eq:C}) and (\ref{eq:CNull}), where the null space was selected to be constant images. We chose the second constraint to see if we would find a different operator by enforcing a similar null-space to the Finite Difference (FD) operator.

The learned operators are shown in Figure \ref{LearnedOperatorRows}. We found operators which have many FD type elements. The operator which we learned using $\mathcal{C}$ in (\ref{eq:CNull}) seems to have more FD type elements. 
We also compare the level of objective value for the learned operator, using $\mathcal{C}$ in (\ref{eq:C}), and a reference operator. As a two dimensional FD operator is not a UNTF, we chose a two-times overcomplete Haar wavelet---see the right panel of Figure \ref{LearnedOperatorRows}---which has some similarities with the FD, in the fine scale. The objective of (\ref{eq:l1l2stC}) for the learned operator in different steps of training and the Haar based operator are shown in Figure \ref{costComp}. This shows that the learned operator finally outperforms the baseline operator in the objective value. Note that if we initialise the algorithm with the Haar based operator, the AOL algorithm does not provide a new operator. This empirically shows that the reference operator is a local optimum for the 
problem (\ref{eq:l1l2stC}), where $\lambda \rightarrow \infty$. This fact can be investigated using the analysis provided in Appendix \ref{app:theory}, \textit{i.e.} randomly generating admissible vectors and checking (\ref{eq:iffnoiseless}) or (\ref{eq:identifiability}).

\begin{figure}[t]
\centering
\vspace{-2mm}
\centerline{\epsfig{figure=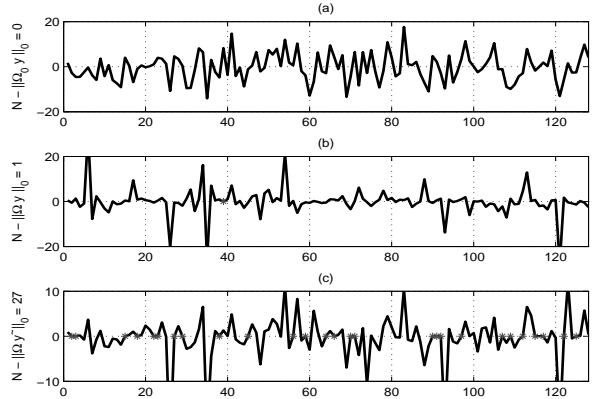,width=9.0cm, height=6cm}}
\vspace{-5mm}
\caption{Signals in the analysis space $\R^{N}$, $N=128$. The
coefficients with almost zero magnitude, \textit{i.e.} less than 0.01, are
indicated with stars. The cosparsity in each case is: (a)
$p=0$, (b) $p=1$, and (c) $p=27$.}
\label{SingleCosparsity}
\end{figure}

\begin{figure}[t]
\centering
\vspace{-2mm}
\centerline{\epsfig{figure=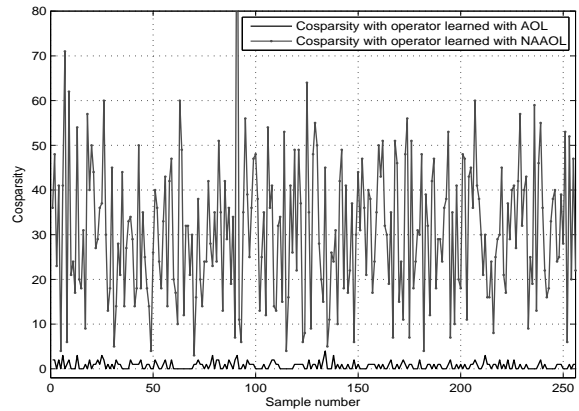,width=9cm, height=6cm}}
\vspace{-5mm}
\caption{The cosparsities of $\y$ (bottom plot) and $\widehat{\y}$ (top plot) respectively with the operators learned with (NL)AOL and NAAOL.}
\label{CosparsityComp}
\end{figure}


\subsection{Image denoising with learned analysis operators}
For the next experiments we used  a set of face images which are centred and cropped \cite{Lee05}. Such images can  be modelled as approximately piecewise smooth signals. A pseudo-random admissible $\Omg^{[0]} \in \R^{128 \times 64}$ has been used as an initial analysis operator and a training set of size $l = 16384$ of $8\times8$ image patches from 13 different faces, after normalised to have mean zero, have been used to learn the operators ($a=128$).

We applied both of the noiseless AOL Algorithm and the noise-aware AOL Algorithm to demonstrate how much the cosparsities of the training samples increase using the noise aware formulation. (NL)AOL algorithm was iterated $K_{\max} = 100000$ times and NAAOL algorithm iterated for $K_{\max} = 10$ iterations with $\lambda = \gamma = 0.5$, while the inner-loop, \textit{i.e.} Algorithm \ref{alg:pcdl}, was iterated 100000 times. 

A plot of the analysis coefficients for a selected $\y$ along with its corresponding cosparsities, with three different $\mathbf{\Omega}$'s, are presented in Figure \ref{SingleCosparsity}. The initial operator $\Omg_0 = \Omg^{[0]}$ has been applied to $\y$ in $(a)$. Not surprisingly, the signal is not cosparse with this arbitrary
operator ($q=0$). In $(b)$, the same plot is drawn using the learned operator with (NL)AOL. Although some coefficients are small, most are not zero, and $q=1$.  In the last plot, we have shown the analysis coefficients for $\x$ using the learned operator with NAAOL. It is clear that the cosparsity has been increased significantly (from $q=1$ to $q=27$). We have further plotted the cosparsities of the first $256$ training samples $\y$'s using the learned operator found by (NL)AOL and corresponding approximations $\x$'s, which are found by NAAOL, in Figure \ref{CosparsityComp}. This figure also shows, the operator learning using the noise aware formulation (\ref{eq:l1l2stC}), where $\lambda$ is finite, results in much greater cosparsity. We show the learned operator found using NAAOL algorithm in Figure \ref{NoisyFaceOperator}. This experiment suggests that a harmonic type operator should perform better than FD based operators for such images. This is an interesting observation which should be explored in the future in more detail.

The aim of next experiment is to compare denoising performance of an image using a learned operator and a FD operator \cite{Rudin92} which is similar to the analysis operator in the TV denoising formulation and is shown in the right panel of Figure \ref{NoisyFaceOperator}). We keep the previous experiment settings. The learned operator and the FD operator can now be used to denoise a corrupted version of another face from the database, using (\ref{eq:nl1}). The original face is shown in Figure \ref{FaceDenoising} (a) and the noisy version, with additive \textit{i.i.d} Gaussian noise, is shown in (b), with PSNR = 26.8720 dB. Denoising was performed using two different regularisation settings: $(\lambda = \gamma ) = 0.3 , \; 0.1$.  The bottom two rows show the denoised images using the FD operator and the learned operator. We can visually conclude that the two operators successfully denoise the corrupted images with some slight differences. The results with the learned operators are smoother (this is 
mostly visible on a screen rather than a printed copy of the paper). The average PSNR's of the denoised images over 100 trials are presented in Table \ref{tab:psnr}. Although we get a marginally better average PSNR using the learned operator instead of FD when $\lambda = 0.3$, we get a noticeably better average PSNR, \textit{i.e.}, 1 dB, when $\lambda = 0.1$.

As the initial goal was to increase the cosparsities of signals, we have also shown the cosparsities of different patches of the selected face image (see Figure \ref{FaceCosparsity}). The horizontal axis presents the index number of the patches. To compare the cosparsity using these operators, we have plotted their differences and its average in the bottom plot. Positive values here demonstrate the cases when the learned operator is a better operator than the FD operator. The average, which is indeed positive, is plotted as a horizontal line. As a result, the learned operator here performs much better (15\% improvement) than the FD operator.

The synthesis framework has been used for different applications with some very promising results. In the last experiment, we have a comparative study between the two frameworks, for denoising face images. We chose the settings of the previous experiment and used the patch based learning. For the reference, we used a two times overcomplete DCT, which is a standard selection for the sparsity based denoising \cite{Elad06b}. For the synthesis sparse representation, we used OMP method, and for the analysis sparse recovery, we used a convex formulation similar to the previous experiment. The optimum value of $\lambda$ in the $\ell_1$ analysis was $0.04$. The OMP was running to achieve an approximation error equal to $1.15 \sigma$, where $\sigma$ is the standard deviation of the additive noise. This setting has been used in \cite{Elad06b}, for denoising with the learned synthesis dictionary using K-SVD method. We used another technique, which has also been used in the mentioned reference, to average out over the overlapping-blocks of images to reduce the blocking effect. The average PSNR's of the denoised face images over 5 trials, for the DCT dictionary/operator are shown in the first row of Table \ref{tab:psnrComp}.

This experiment demonstrates that the synthesis framework works better in the denoising of face images. We already have techniques to learn dictionaries/operators for each of these three cases. For the OMP and analysis based denoisings, we respectively chose K-SVD \cite{Aharon06} and the proposed CAOL methods. The Lagrange multiplier for the analysis based denoising was $\lambda = 0.09$, while we used the operator learned in the previous experiment. The average PSNR's of the denoised face images over 5 trials are presented in the second row of Table \ref{tab:psnrComp}. 

By comparing the synthesis denoising techniques, with the overcomplete DCT and K-SVD dictionaries, we find a slight improvement (+0.2 dB) in the PSNR of the denoised images, using the learned dictionary. This improvement is more significant in the analysis denoising framework, as we achieved more than 2.5 dB improvement in the PSNR of the denoised image. In comparison between two frameworks, although the PSNR of the denoised image using the learned operator is significantly improved, it is marginally behind the synthesis competitor. However, we found these results very promising, since this is only the beginning of the field ``denoising with the learned analysis operators''. More experiments are necessary to find a reason for such a behaviour, which we leave it for the future work.

\begin{figure}[t]
\centering
\centerline{\epsfig{figure=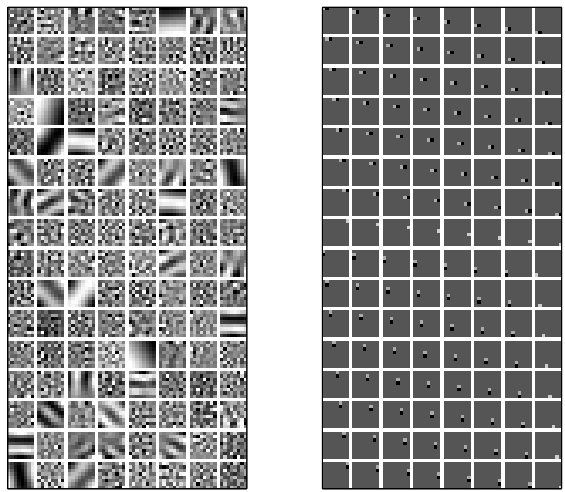,width=9.5cm}}
\vspace{-5mm}
\caption{The learned analysis operator in a noisy setting, $\lambda = 0.1$ (left panel) and the finite difference operator (right panel).}
\label{NoisyFaceOperator}
\end{figure}


\begin{figure}[t]
\centering
\centerline{\epsfig{figure=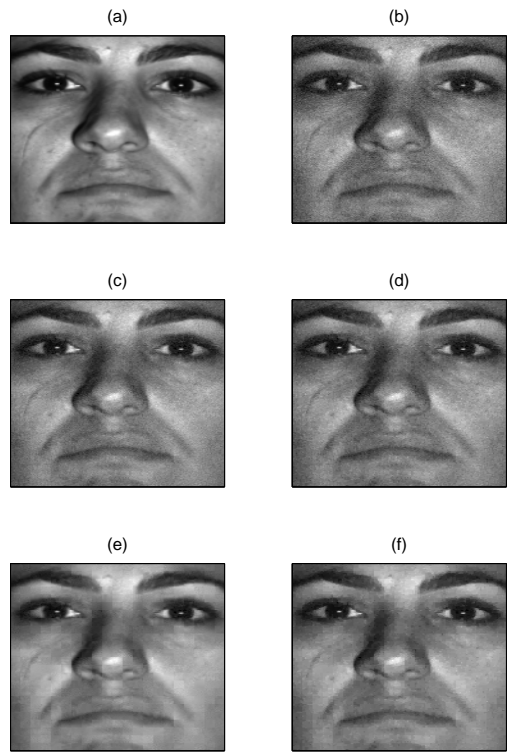,width=8.5cm}}
\vspace{-13mm}
\caption{Face image denoising. Top row: original face (left), noisy
  face (right). Denoising results using (\ref{eq:nl1}). Middle row:
  $\lambda = 0.3$ using the learned analysis operator (left) and  the finite difference operator 
  (right). Bottom row: same as middle row with $\lambda = 0.1$.}
\label{FaceDenoising}
\end{figure}

\begin{figure}[ht]
\centering
\centerline{\epsfig{figure=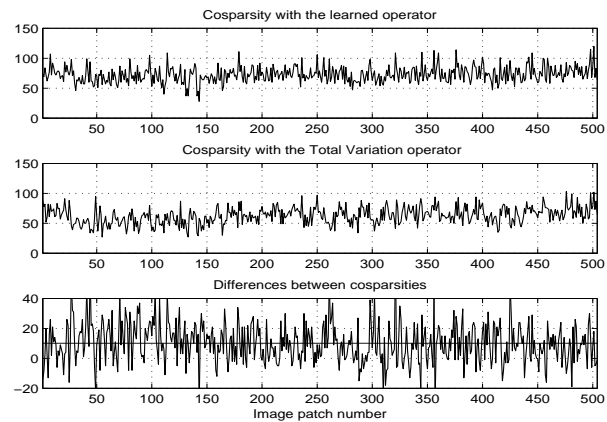,width=9.5cm, height=6cm}}
\vspace{-5mm}
\caption{Cosparsities of image patches (a)-(b) and a comparison (c).}
\label{FaceCosparsity}
\end{figure}

\begin{table}[t]
\caption{Average PSNR (dB) of the denoised face images (100 trials) using, a) FD and b) the learned operators, when the average PSNR of the given noisy images was 26.8720 dB}\label{tab:psnr}
\centering
\begin{tabular}{ l | c  c}
	&	$\lambda = 0.3$	&	$\lambda = 0.1$ \\
\hline
\\
Finite Difference	&	28.9643	&	31.5775 \\
Learned Operators		&	29.0254	&	32.5202 \\
\end{tabular} 
\end{table}

\begin{table}[t]
\caption{Average PSNR (dB) of the denoised face images (5 trials), when the average PSNR of the given noisy images was approximately 26.6 dB (varies slightly in different experiments)}\label{tab:psnrComp}
\centering
\begin{tabular}{| l  | c | c |}
\hline
	&		Synthesis (OMP)  &  Analysis ($\ell_1$)\\
\hline
DCT	&	34.63 & 	31.6  \\
\hline
Learned 	&	 34.87  &	 34.22 \\
\hline 
\end{tabular} 
\end{table}

\section{Summary and Conclusion}

In this paper, we presented a new concept for learning linear analysis operators, called constrained analysis operator learning. The need for a constraint in the learning process comes from the fact that we have various trivial solutions, which we want to avoid by selecting an appropriate constraint. A suitable constraint for this problem was introduced after briefly explaining why some of the canonical constraints in model learning problems are not sufficient for this task. Although there is no claim that the introduced constraint is the most suitable selection, we practically observed very good results with the introduced algorithm. 

In the simulation part, we showed some results to support our statements about the relevance of the constraint and the learning algorithm. We actually demonstrated that the learning framework can recover a synthetic analysis operator, when the signals are cosparse and enough training signals are given. By applying the algorithm on two different types of image classes, \textit{i.e.} piecewise constant and approximately piecewise smooth images, we learned operators which have respectively similarities with the finite difference and harmonic type operators. This observation emphasises the fact that we benefit from selecting the most appropriate analysis operator for image processing tasks. The Matlab code of proposed simulations will be available, as a part of the dictionary learning toolbox SMALLbox \cite{Damnjanovic10}, later.

The proposed constrained optimisation problem is a non-convex program, for which we can often find a local optimum using variational optimisation methods. In the appendix, we characterise the local optima of such programs. This can be useful to avoid initialising the algorithm with a local minima, and also for empirically checking when an operator can be a local minima. The latter emphasises the fact that we can recover such an operator, by starting the algorithm with a point in a close neighbourhood of the operator, which we do not know \textit{a priori}.

\appendices
\section{Theoretical Analysis of the CAOL}\label{app:theory}


If we rewrite the optimisation program (\ref{eq:l1l2stC}), using the UNTF admissible set, we find the following program,

\begin{equation}\label{eq:optformulation}
\begin{split}
 \min_{\mathsmaller{\Omg}, {\X}} \|\Omg \X\|_1 + \frac{\lambda}{2} \|\Y -  \X \|_{\mathsmaller{F}}^2 \ & \subto  \ \Omg^T \Omg = \I \\
			& \ \ \forall_i \ \|\omg_i\|_2 = c . 
\end{split}
\end{equation}

The variational analysis \cite{Rockafellar97} of the objective, near to a given pair $(\Omg_0,\X_0)$, provides the following proposition.

\begin{proposition}\label{prop:linearized}
A given pair $(\Omg_0,\X_0)$ is a local minimum for (\ref{eq:optformulation}) if and only if $\Delta = \mathbf{0}$ and $\Sigma = \mathbf{0}$ are the only solutions of the following \emph{convex} program,
\begin{equation}\label{eq:linearized}
\begin{split}
 \min_\mathsmaller{\Delta, \Sigma} &\overbrace{\| \Omg_0 \X_0 + \Delta \X_0 + \Omg_0 \Sigma\|_1 + \lambda \tr\{\Sigma^T (\X_0 - \Y)\}}^{\mathcal{L}_{\lambda}(\Delta,\Sigma)} \\ 
		    \subto \ & \Delta^T \Omg_0 + \Omg_{0}^T \Delta = \mathbf{0}  \\
			     & \forall i \ \left< \omg_{0_i}, \delta_i\right> = 0.
\end{split}
\end{equation}
\end{proposition}

Proposition \ref{prop:linearized} is actually checking the local optimality of a pair, using the first-order approximation of the objective, which can be achieved using $\X \leftarrow \X_0 + \Sigma$ and $\Omg \leftarrow \Omg_0 + \Delta$. The constraint of (\ref{eq:linearized}) is the tangent space of the constraint of (\ref{eq:optformulation}) at $\Omg_0$,
\begin{equation*}
 \mathcal{T}_C := \{ \Delta : \Delta^T \Omg_0 + \Omg_{0}^T \Delta = \mathbf{0} \ \& \ \forall i \ \left< \omg_{0_i}, \delta_i\right>  = 0\},
\end{equation*}
which is a \emph{linear} subspace of $\R^{a \times n}$.

This proposition will be used in the next two subsections to investigate the conditions for the local optimality of a point in two different scenarios: (NA)AOL and NLAOL.

\subsection{Noise Aware Analysis Operator Learning}

The tightest condition for the local optimality of pair $(\Omg_0,\X_0)$, is the following \textit{necessary and sufficient} condition.

\begin{lem}\label{lem:iffnoisy}
 A pair $(\Omg_0,\X_0)$ is a local optimum of (\ref{eq:optformulation}) if and only if the following inequality holds,
\begin{equation}\label{eq:optiSign}
\begin{split}
\|(\Delta \X_0 + \Omg_0 \Sigma)_{\bar{\Lambda}} \|_1 > \ &|\left< \Delta \X_0 + \Omg_0 \Sigma, \sgn(\Omg_0 \X_0) \right> \\
						       & +  \lambda \tr\{\Sigma^T (\X_0 - \Y)\}|,
\end{split}
\end{equation}
for all non-zero $(\Delta,\Sigma) \in \mathcal{T}_C  \times \R^{n \times l}$.
\end{lem}
\begin{proof}
 According to the proposition \ref{prop:linearized}, we can check the optimality of $(\Omg_0,\X_0)$ by checking that $(\Delta, \Sigma) = (\mathbf{0}, \mathbf{0})$ is the only solution of (\ref{eq:linearized}). The proof here is formed by showing that the objective of (\ref{eq:linearized}) increases if $(\Omg_0,\X_0)$ is replaced with $(\Omg_0 + \Delta,\X_0 + \Sigma)$, or equivalently,  $\lim_{t\downarrow0} \frac{1}{t} (\mathcal{L}_{\lambda}(t\Delta, t\Sigma) - \mathcal{L}_{\lambda}(\mathbf{0},\mathbf{0})) > 0, \ \forall \Sigma, \forall\Delta \in \mathcal{T}_C$\footnote{This condition for the (local) optimality is related to the positivity condition of directional derivative of $\mathcal{L}_{\lambda}$ for any admissible $\Delta$ and $\Sigma$ in the tangent cone(s) of the constraint(s), see for example \cite{Rockafellar97}.}, when (\ref{eq:optiSign}) is assured. From the definitions of $\mathcal{L}_{\lambda}(\Delta, \Sigma)$ in (\ref{eq:linearized}) and $\|\mathbf{A}\|_1 = \left<\mathbf{A},\sgn{\mathbf{A}} \right>$, 
we have,
\begin{equation}\label{eq:directionalDerivative}
 \begin{split}
  \mathcal{L}_{\lambda} & (t\Delta, t\Sigma) - \mathcal{L}_{\lambda}(\mathbf{0},\mathbf{0}) \\ 
=& \|\Omg_0 \X_0 + t(\Delta \X_0 + \Omg_0 \Sigma)\|_1  - \|\Omg_0 \X_0\|_1 + t\lambda \tr\{\Sigma^\mathsmaller{T} \E\} \\
=& \left< \Omg_0 \X_0 + t(\Delta \X_0 + \Omg_0 \Sigma), \sgn(\Omg_0 \X_0 + t(\Delta \X_0 + \Omg_0 \Sigma)) \right> \\
& - \left< \Omg_0 \X_0, \sgn(\Omg_0 \X_0) \right> + t\lambda \tr\{\Sigma^\mathsmaller{T} \E\}\\
\overset{t\downarrow 0}{=}& t\left< \Delta \X_0 + \Omg_0 \Sigma, \sgn(\Omg_0 \X_0)\right> \\
+& t\left< \Delta \X_0 + \Omg_0 \Sigma, \sgn((\Delta \X_0 + \Omg_0 \Sigma)_{\bar{\Lambda}})\right> + t\lambda \tr\{\Sigma^\mathsmaller{T} \E\}  \\
=& t\left< \Delta \X_0 + \Omg_0 \Sigma, \sgn(\Omg_0 \X_0)\right> \\
+& t\|(\Delta \X_0 + \Omg_0 \Sigma)_{\bar{\Lambda}}\|_1 + t\lambda \tr\{\Sigma^\mathsmaller{T} \E\}  \\
 \end{split}
\end{equation}
where $\E := \X_0 - \Y$. The positivity of (\ref{eq:directionalDerivative}) is assured when
\begin{equation*}
 \begin{split}
  \|(\Delta \X_0 + \Omg_0 \Sigma)_{\bar{\Lambda}}\|_1 >& \ |\left< \Delta \X_0 + \Omg_0 \Sigma, \sgn(\Omg_0 \X_0)\right> \\
						      &+ \lambda\tr\{\Sigma^\mathsmaller{T}\E\}| .
 \end{split}
\end{equation*}
This completes the sufficiency of (\ref{eq:optiSign}) for the optimality. If (\ref{eq:optiSign}) is violated, $\exists (\Delta ,\Sigma) \in \mathcal{T}_C \times \R^{m \times L}$ such that $|\left< \Delta \X_0 + \Omg_0 \Sigma, \sgn(\Omg_0 \X_0)\right> + \lambda\tr\{\Sigma^\mathsmaller{T}\E\}|$ is greater than $\|(\Delta \X_0 + \Omg_0 \Sigma)_{\bar{\Lambda}}\|_1$. As the constraint $\mathcal{T}_C$, is a linear space, when $\Delta$ and $\Sigma$ are admissible, $\pm \Delta$ and $\pm \Sigma$ are also admissible. If we set $s = \sgn(\left< \Delta \X_0 + \Omg_0 \Sigma, \sgn(\Omg_0 \X_0)\right> + \lambda\tr\{\Sigma^\mathsmaller{T}\E\}) \in \{-1,1\}$ and $(\Delta_{wc},\Sigma_{wc}) = (s \Delta,s \Sigma) \in \mathcal{T}_C \times \R^{n \times l}$, it is easy to show that  
\begin{equation*}
\begin{split}
&\left< \Delta_{wc} \X_0 + \Omg_0 \Sigma_{wc}, \sgn(\Omg_0 \X_0)\right> + \lambda \tr\{\Sigma_{wc}^\mathsmaller{T} \E\}\\
&+\left< \Delta_{wc} \X_0 + \Omg_0 \Sigma_{wc}, \sgn((\Delta_{wc} \X_0 + \Omg_0 \Sigma_{wc})_{\bar{\Lambda}})\right> < 0 ,
 \end{split}
\end{equation*}
which contradicts the optimality of $(\Omg_0,\X_0)$, \textit{i.e.} $\exists (\Delta,\Sigma) \in \mathcal{T}_C  \times \R^{n \times l}, \ \lim_{t\downarrow0} \frac{1}{t} (\mathcal{L}_{\lambda}(t\Delta, t\Sigma) - \mathcal{L}_{\lambda}(\mathbf{0},\mathbf{0})) > 0$. This completes the necessary part of the lemma.
\end{proof}

\begin{remark}
 As lemma \ref{lem:iffnoisy} is valid for any non-zero admissible pair $(\Delta, \Sigma)$, we can choose $(\Delta, \mathbf{0})$ and $(\mathbf{0},\Sigma)$ and get the following \textit{necessary} conditions for optimality,
\begin{equation*}
 \|(\Delta \X_0)_{\bar{\Lambda}}\|_1 > \ |\left< \Delta \X_0, \sgn(\Omg_0 \X_0)\right>|, \ \forall \Delta \in \mathcal{T}_C ,
\end{equation*}
\begin{equation*}
 \|(\Omg_0 \Sigma)_{\bar{\Lambda}}\|_1 > \ |\left< \Omg_0 \Sigma, \sgn(\Omg_0 \X_0)\right> + \lambda\tr\{\Sigma^\mathsmaller{T}\E\}| .
\end{equation*}

\end{remark}

\begin{lem}\label{lem:Noisysufficient}
 A \textit{sufficient} condition for the local optimality of a pair $(\Omg_0, \X_0)$ for (\ref{eq:optformulation}) is,
\begin{equation} \label{eq:noisysuff}
 \begin{split}
\|(\Delta \X_0 + \Omg_0 \Sigma)_{\bar{\Lambda}} \|_1 > \ & \| (\Delta \X_0 + \Omg_0 \Sigma)_{\Lambda} \|_1 \\
						       & +  \lambda|\tr\{\Sigma^T(\X_0 - \Y)\}|,
\end{split}
\end{equation}
for all $(\Delta,\Sigma) \in \mathcal{T}_C \times \R^{n \times l}$.
\end{lem}

\begin{proof}
 The proof is based on finding an upper bound for the right hand side of (\ref{eq:optiSign}). Note that the maximum of $\left< \Delta \X_0 + \Omg_0 \Sigma, \sgn(\Omg_0 \X_0) \right>$  may be achieved when $\sgn(\Delta \X_0 + \Omg_0 \Sigma)$ is equal to $\sgn(\Omg_0 \X_0)$ on $\Lambda$. Thus,
\[
 \left< \Delta \X_0 + \Omg_0 \Sigma, \sgn(\Omg_0 \X_0) \right> \le \| (\Delta \X_0 + \Omg_0 \Sigma)_{\Lambda} \|_1. 
\]
This means that condition \eqref{eq:noisysuff} implies the condition of Lemma~\ref{lem:iffnoisy} and $(\Omg_0,\X_0)$ is thus a local minimum for (\ref{eq:optformulation}). 
\end{proof}
\begin{remark}
 Note that lemma \ref{lem:Noisysufficient} only presents a \textit{sufficient} condition, as the sign pattern match \textbf{may} not happen, \textit{i.e.} $\sgn(\Delta \X_0 + \Omg_0 \Sigma) \ne \sgn(\Omg_0 \X_0), \ \forall \Delta \in \mathcal{T}_C, \Sigma$.
\end{remark}
\begin{remark}
 Although the recovery condition here has some similarities with the dictionary recovery conditions \cite{Gribonval10,Geng11}, the analysis learning formulation is slightly more involved as it is not possible to check the local optimality based on each parameter, $\Omg$ or $\X$, separately. This is caused by the fact that the non-differentiable term, \textit{i.e} $\ell_1$, is a function of both parameters. Here, it makes the problem non-separable and the joint optimality condition is needed to be checked.
\end{remark}

\subsection{Noiseless Analysis Operator Learning}\label{sec:naolt}
When the noise and model mismatch does not exist, we can solve (\ref{eq:optformulation}) with $\lambda \rightarrow \infty$. This case may be considered as an ideal operator learning form, as $\Y$ is here exactly cosparse. This formulation can be used as a benchmark for the operator recovery, very similar to the framework in \cite{Gribonval10,Geng11} for dictionary recovery.
In this setting we have $\X = \Y \in \R^{n \times l}$, we can simplify the problem as follows,
\begin{equation}\label{eq:noislessoptformulation}
 \begin{split}
 \min_{\mathsmaller{\Omg}} \|\Omg \X\|_1 \ & \subto  \ \Omg^T \Omg = \I \\
			& \ \ \forall_i \ \|\omg_i\|_2 = c.
\end{split}
\end{equation}
As this formulation has a single parameter $\Omg$, the analysis is significantly easier. (\ref{eq:noislessoptformulation}) has $n \times a$, \textit{i.e.} $\Omg_{a \times n}$, unknown parameters and $n^2 + a$ constraints. For a full rank $\X$, the system of equations is underdetermined if $na > n^2 +a$. As we assume $\X$ is \emph{rich enough} and therefore full rank, minimising the objective $\|\Omg \X\|_1$ is reasonable if $a > n^2 / (n-1)$. 

We use formulation (\ref{eq:noislessoptformulation}) to recover operator $\Omg_0$, while $\X$ is cosparse with $\Omg_0$. The following proposition investigate the local optimality conditions of a point $\Omg_0$ for (\ref{eq:noislessoptformulation}), which is actually equivalent to Proposition \ref{prop:linearized}, when $\lambda \rightarrow \infty$. This proposition has a flavor similar to those obtained in the context of dictionary learning for the synthesis model~\cite{Gribonval10}, as well as in hybrid synthesis/analysis framework~\cite{Jaillet10}.

\begin{proposition}\label{pro:localiden}
 An operator $\Omg_0$ is locally identifiable using (\ref{eq:noislessoptformulation}), if and only if $\Delta = 0$ is the only solution of the following program,
\begin{equation}\label{eq:noiselesslinerized}
\begin{split}
 \min_\mathsmaller{\Delta} \| (\Omg_0 + \Delta) \X \|_1  \ 
		    \subto \ &\Delta^T \Omg_0 + \Omg_{0}^T \Delta = \mathbf{0}  \\
			     & \forall i \ \left< \omg_{0_i}, \delta_i\right> = 0.
\end{split}
\end{equation}
\end{proposition}
Note that (\ref{eq:noiselesslinerized}) has a convex objective with a linear constraint. This matrix valued convex problem has a close relation to the analysis parsimony problems investigated in \cite{Elad07c,Candes11,Nam11b,Vaiter11}, where $\X^\mathsmaller{T}$ acts as the analysis operator of the columns of $(\Omg_0 +\Delta)^T$. We present some conditions for locally identifiability of a $\Omg_0$ using variational analysis. A similar technique for the analysis parsimony problems have already been used in \cite{Nam11b,Vaiter11}. 

The kernel of the matrix $\X$ has an important role in the identifiability of the analysis operator. Let $\X_{n \times l} = \U_{n \times n} \Sigma_{n \times l} \V_{l \times l}^\mathsmaller{T}$ be a singular value decomposition of $\X$. We can now partition $\V$ into two parts, \textit{i.e.} one part $\V_{1}^\mathsmaller{T}$ is multiplied to $\Sigma_1 = \operatorname{diag}(\sigma_i) \in \R^{n \times n}$, which is the non-zero diagonal part of $\Sigma$, and $\V_0$, a basis for the kernel of $\X$, and find $\X = \U \Sigma_{1} \V_{1}^\mathsmaller{T}$ and $\X^\mathsmaller{\dagger} = \V_{1} \Sigma_{1}^\mathsmaller{-1} \U^\mathsmaller{T}$. Let two new parameters be defined as $\Z_0 := \Omg_0 \X \ , \ \Delta_z := \Delta \X \ , \in \R^{a \times l}$. We can now reformulate (\ref{eq:noiselesslinerized}) based on the new parameters as follows,
\begin{equation}\label{eq:noiselesslinerizedreparametrized}
\begin{split}
 \min_\mathsmaller{\Delta_z} \| \Z_0 + \Delta_z \|_1  \ 
		    \subto \  \V_{1}^\mathsmaller{T}\left(\Delta_{z}^\mathsmaller{T} \Z_0 + \Z_{0}^\mathsmaller{T} \Delta_{z}\right) \V_{1} &= \mathbf{0}  \\
			      \Delta_{z} \left(\V_{1} \Sigma_{1}^\mathsmaller{-2} \V_{1}^\mathsmaller{T} \right) \Z_{0}^\mathsmaller{T} \circ \I &= \mathbf{0} \\
			      \Delta_{z} \V_{0} &= \mathbf{0}
\end{split}
\end{equation}
where $\circ$ is the Hadamard product. The constraint of (\ref{eq:noiselesslinerizedreparametrized}) is linear, which it means, we can represent it based on a linear operator $\mathbf{\Psi}: \R^{al} \rightarrow \R^{n^2 + a + a(l-n)}$ and the vector version of $\Delta_{z}$, $\delta_{z}^\mathsmaller{v} = \operatorname{vect}(\Delta_{z})$ as $\mathbf{\Psi} \delta_{z}^\mathsmaller{v} = \mathbf{0}$. Such an operator is derived in Appendix \ref{app:phi}. $\mathbf{\Psi}$ has a nontrivial kernel when $a > n^2 / (n-1)$. We can now characterise the kernel subspace of the constraint of (\ref{eq:noiselesslinerizedreparametrized}), $\mathcal{T}_{C_s}$, as follows,
\begin{equation*}
 \mathcal{T}_{C_s} = \{\Delta \in \R^{a \times l} : \mathbf{\Psi} \delta^\mathsmaller{v} = \mathbf{0}, \delta^\mathsmaller{v} = \vect(\Delta) \}
\end{equation*}

\begin{lem}\label{lem:iffnoiseless}
 An operator $\Omg_0$ is identifiable using the cosparse training samples $\X$, where $\forall i:  \operatorname{supp}(\Omg \x_i) = \lambda_i$, and (\ref{eq:noislessoptformulation}), if and only if for all $\Delta_{z} \in \mathcal{T}_{C_s}$,
\begin{equation}\label{eq:iffnoiseless}
 |\left< \sgn(\Omg_0 \X), \Delta_z  \right>| < \|(\Delta_z)_{\bar\Lambda}\|_1
\end{equation}
where $\bar\Lambda$ is the cosupport of $\X$, \textit{i.e.} $\bar\Lambda  =  \{(i,j): (\Omg_0 \X)_{i \: j} = 0\}$.
\end{lem}

\begin{proof}
As (\ref{eq:noiselesslinerized}) is a convex problem, $\Delta = 0$ is the only solution iff $\|(\Omg_0 + t\Delta)\X_0\|_1 - \|\Omg_0\X_0\|_1 > 0, \ \forall \Delta \in \mathcal{T}_{C_s}$ and $t \downarrow 0$. This is equivalent to the statement based on (\ref{eq:noiselesslinerizedreparametrized}) as follows, $\Delta = \Delta_z \X^{\dagger} = \mathbf{0}$ is the only solution of (\ref{eq:noiselesslinerized}) iff $\|\Z_0 + t\Delta_z\|_1 - \|\Z_0\|_1 > 0, \ \forall \Delta_z \in \mathcal{T}_{C_s}$ and $t \downarrow 0$. We can now find the necessary and sufficient condition for the subject as follow,
\begin{equation*}
 \begin{split}
  \|\Z_0 \!+t\Delta_z\|_1 \!- \|\Z_0\|_1 &= \left< \Z_0 \!+t\Delta_z, \sgn(\Z_0 \! +t\Delta_z) \right> \! - \|\Z_0\|_1\\
				    &\overset{t\downarrow 0}{=} \left< \Z_0 +t\Delta_z, \sgn(\Z_0) \right> \\
				    & \quad + t \left< \Delta_z, \sgn(\Delta_z) \right> \! - \|\Z_0\|_1\\
				    &= t \left< \Delta_z, \sgn(\Z_0) \right> + t \| (\Delta_z)_{\bar\Lambda}\|_1
 \end{split}
\end{equation*}
As $t>0$, positivity of $ \|\Z_0 \!+t\Delta_z\|_1 \!- \|\Z_0\|_1 $ is equivalent to the positivity of $\left< \Delta_z, \sgn(\Z_0) \right> +  \| (\Delta_z)_{\bar\Lambda}\|_1$. $\mathcal{T}_{C_s}$ is a linear space and therefore if $\Delta_z \in \mathcal{T}_{C_s}$, $-\Delta_z \in \mathcal{T}_{C_s}$ too. This freedom in choosing the sign of $\Delta_z$, lets us have $\left< \Delta_{z}', \sgn(\Z_0) \right>  = -|\left< \Delta_z, \sgn(\Z_0) \right>|$ for any given $\Delta_{z}$, such that $|\left< \Delta_z, \sgn(\Z_0) \right>| = |\left< \Delta_{z}', \sgn(\Z_0) \right>|$. Therefore, the necessary and sufficient condition for $\Delta = \mathbf{0}$ to be the only solution of (\ref{eq:noiselesslinerized}) is $\forall \Delta_z \in \mathcal{T}_{C_s}$, $|\left< \Delta_z, \sgn(\Omg_0 \X) \right>|  <  \| (\Delta_z)_{\bar\Lambda}\|_1$
\end{proof}

\begin{theorem}\label{theo:identifiability}
 Any analysis operator $\Omg_0 : \R^n \rightarrow \R^a$ is identifiable by (\ref{eq:noislessoptformulation}), where  $\X = [\x_i]_\mathsmaller{i \in [1,l]}$, such that $\forall i: \operatorname{supp}(\Omg \x_i) = \lambda_i,\ a - |\lambda_i| \ge q $, is a training $q$-cosparse matrix, if $\forall \Delta_{z} \in \mathcal{T}_{C_s}$,
\begin{equation}\label{eq:identifiability}
  \| (\Delta_z)_{\Lambda}\|_1 <  \| (\Delta_z)_{\bar\Lambda}\|_1
\end{equation}
where $\bar\Lambda = \Lambda^{c}$.
\end{theorem}
\begin{proof}
This theorem is a generalisation of lemma \ref{lem:iffnoiseless}. The proof is based on, firstly, finding a lower-bound for the left-hand side of (\ref{eq:iffnoiseless}) and, secondly, generalising the lemma using any \textit{possible} cosparsity pattern $\bar\Lambda$. 

$\left< \sgn(\Omg_0 \X), \Delta_z  \right>$ is never greater than $\left< \sgn(\Delta_z), \Delta_z  \right>$, as,
\begin{equation*}
 \begin{split}
  \left< \sgn(\Omg_0 \X), \Delta_z  \right> & = \sum_{(i,j) \in \Lambda} |(\Delta_z)_{i,j}| \sgn(\Delta_z)_{i,j} \sgn(\Omg_0 \X)_{i,j}\\
					 & \le \sum_{(i,j) \in \Lambda} |(\Delta_z)_{i,j}| \\
					 & = \sum_{(i,j) \in \Lambda} (\Delta_z)_{i,j} \sgn(\Delta_z)_{i,j} \\
					 & = \left< \sgn(\Delta_z)_{\Lambda}, \Delta_z  \right> \\
 \end{split}
\end{equation*}
This assures that $\left< \sgn(\Omg_0 \X), \Delta_z  \right> \le \|(\Delta_z)_\Lambda\|_1$. For any $\Omg_0$, it assures that if $ \|(\Delta_z)_\Lambda\|_1 < \|(\Delta_z)_{\bar\Lambda}\|_1$, $\Omg_0$ is identifiable based on lemma \ref{lem:iffnoiseless}. Now, if we consider all $\Omg_0$ providing $q$-cosparse training samples, with some \textit{possible} cosparsity $\bar\Lambda$, the inequality (\ref{eq:identifiability}) assures the identifiability of $\Omg_0$.
\end{proof}
\begin{remark}
 Note that, not all index sets of cardinality $q$ can produce cosparsity patterns $\bar\Lambda$, given $\Omg_0$ \cite{Nam11b}. Similarly, not all sign patterns are feasible as $\sgn(\Omg_0 \X)$, when $\X$ is a variable matrix. As a result, the inequality condition of Theorem \ref{theo:identifiability}, \textit{i.e.} (\ref{eq:identifiability}), is \textit{just a sufficient condition} for the local identifiability of any $\Omg_0$'s, providing some $q$-cosparse training signals. While (\ref{eq:iffnoiseless}) is the necessary and sufficient condition for the local identifiability, we can empirically show that (\ref{eq:identifiability}) is not tight, see Figure \ref{identifiability}.
\end{remark}

\section{Deriving a simple linear representation of the constraints}\label{app:phi}

The optimisation problem (\ref{eq:noiselesslinerizedreparametrized}) has three constraints, where the last two are clearly linear, and they can be represented as,
\begin{equation}\label{eq:23cons}
 \begin{split}
 &\sum_{k} (\Delta_{z})_\mathsmaller{i\:k} \; (\mathbf{V}_{0})_{\mathsmaller{ k\,j}}= 0, \quad \forall i \in [1,n], \; j \in [1,l-m] \\
 & \sum_{k}  {\Delta_{z}}_\mathsmaller{\;i\,k} \; \mathbf{Q}_{\mathsmaller{\; k\, i}} = 0, \quad \forall i \in [1,n],
 \end{split} 
\end{equation}
 where $\mathbf{Q} := \V_1 \Sigma_{1}^{-2}\V_{1}^\mathsmaller{T} \mathbf{Z}_{\mathsmaller{0}}^\mathsmaller{T}  $. Let $\mathbf{A} := \mathbf{Z}_\mathsmaller{0} \V_1$. The first constraint can now be reformulated as, $\V_{1}^\mathsmaller{T} \Delta_{z}^\mathsmaller{T} \mathbf{A} + \mathbf{A}^\mathsmaller{T} \Delta_{z} \V_1 = \mathbf{0}$. To derive this equation in a similar form to (\ref{eq:23cons}), we reformulate $\Delta_{z} \V_{1} $, then right-multiply $(\Delta_{z} \V_{1})^\mathsmaller{T}$ with $\mathbf{A}$.

\begin{equation*}
 \left(\V_{1}^\mathsmaller{T} \Delta_{z}^\mathsmaller{T}\right)_{i \, j} =  \sum_{\mathsmaller{1\le r \le l}} {(\Delta_{z})}_{j \, r} \: (\V_{1})_{r \, i}  
\end{equation*}
\begin{equation*}
\begin{split}
 ((\V_{1}^\mathsmaller{T} \Delta_{z}^\mathsmaller{T})\, \mathbf{A})_{i \, j} &=  \sum_{\mathsmaller{1\le k \le n}}\; \left(\sum_{\mathsmaller{1\le r\le l}} {(\Delta_{z})}_{k \, r} \: (\V_{1})_{r \, i} \right) \mathbf{A}_{\mathsmaller{k \, j}}\\
&= \sum_{\mathsmaller{1\le k \le n}}\; \sum_{\mathsmaller{1\le r\le l}}  (\V_{1})_{r \, i} \: \mathbf{A}_{\mathsmaller{k \, j}} \, {(\Delta_{z})}_{k \, r} 
\end{split} 
\end{equation*}
We now reformulate the first constraint as,
\begin{equation}\label{eq:1cons}
\begin{split}
 \sum_{\mathsmaller{1\le k \le n}}  \sum_{\mathsmaller{1\le r\le l}} \! \left((\V_{1})_{r \, i} \: \mathbf{A}_{\mathsmaller{k \, j}}  + \, (\V_{1})_{r \, j} \: \mathbf{A}_{\mathsmaller{k \, i}}  \right)& {(\Delta_{z})}_{k \, r} = 0 \\  &\forall i,\,j \in [1,m] 
\end{split}
\end{equation}
We can now generate $\mathbf{\Psi}_\mathsmaller{(a + a (l-n) + n^\mathsmaller{2}) \times al}$ using the weights in (\ref{eq:23cons}) and (\ref{eq:1cons}), corresponding to $\Delta_{z}$, and make a linear presentation as $\mathbf{\Psi}  \delta_{z}^{v}= \mathbf{0}, \: \delta_{z}^{v} = \vect(\Delta_{z})$.

\ifCLASSOPTIONcaptionsoff
  \newpage
\fi

\bibliographystyle{IEEEtran}
\bibliography{biblio.bib}

\end{document}